\documentclass[a4paper,11pt]{amsart}
\usepackage{amsmath}
  \usepackage{paralist}
  \usepackage{graphics} 
  \usepackage{epsfig} 
 \usepackage[colorlinks=true]{hyperref}
\hypersetup{urlcolor=blue, citecolor=red}
\usepackage{amsmath}
\usepackage{amsfonts,amssymb}
\usepackage[all]{xy}
\usepackage{stmaryrd}
\usepackage{hyperref}
\usepackage{enumerate}

\newtheorem{theorem}{Theorem}[section]
\newtheorem{corollary}{Corollary}

\newtheorem{proposition}{Proposition}

\theoremstyle{definition}
\newtheorem{definition}[theorem]{Definition}
\newtheorem{remark}{Remark}

\newtheorem{ex}{Example}

\renewcommand{\d}{\mathrm d}

\newcommand{\X}{\ensuremath{\mathfrak{X}}}

\newcommand{\U}{\ensuremath{\mathcal{U}}}

\newcommand{\C}{\mathcal{C}}
\newcommand{\CV}{\mathcal{C}_V}
\renewcommand{\H}{\mathcal{H}}          
\newcommand{\h}{\mathfrak{h}}           
\newcommand{\Ver}{\mathcal{V}}          
\newcommand{\Hor}{\text{\rm Hor}}       
\newcommand{\Lie}{\mathcal{L}}          
\newcommand{\im}{\text{\rm Im}\,}       
\newcommand{\ad}{\text{\rm ad}\,}       
\newcommand{\g}{\mathfrak{g}}           

\newcommand{\D}{\mathcal{D}}
\newcommand{\loc}{\, \substack{{\scriptscriptstyle loc.}\\{\sim}}\, }

\DeclareMathOperator{\id}{id}
\DeclareMathOperator{\Symp}{Symp}
\DeclareMathOperator{\Der}{Der}
\DeclareMathOperator{\Graph}{Graph}
\DeclareMathOperator{\Curv}{\text{Curv}}

\newcommand{\partialV}{\d^{_V}}
\newcommand{\partialC}{\d^{_C}}
\newcommand{\partialH}{\d^{_H}}
\newcommand{\partialCp}{\d^{_{C'}}}

\makeatletter
\def\Ddots{\mathinner{\mkern1mu\raise\p@
\vbox{\kern7\p@\hbox{.}}\mkern2mu
\raise4\p@\hbox{.}\mkern2mu\raise7\p@\hbox{.}\mkern1mu}}
\makeatother

\newcommand{\comment}[1]{}

\title[Infinitesimal Gauge Symmetries of closed forms]
      {Infinitesimal Gauge Symmetries of closed forms}

\author[Olivier Brahic]{}


 \email{brahic@math.ist.utl.pt}

\thanks{The author is supported by FCT/POCTI/FEDER and by the projects POCI/MAT/55958/2004 and POCI/MAT/57888/2004}

\begin{document}
\maketitle

\centerline{\scshape Olivier Brahic}
\medskip
{\footnotesize
 \centerline{Instituto Superior T\'{e}cnico, 
    dep. de Matem\'{a}tica.Lisboa, Portugal.}
} 


\begin{abstract}
 Motivated by the relationship between symplectic fibrations and classical Yang-Mills theories, we study the closedness of a $n$-form ($n$=2,3) defined on the total space of a fibration as a simple model for an abstract field theory. We introduce $2$-plectic fibrations and interpret geometrically the corresponding equations for coupling in terms of higher analogues of connections.
\end{abstract}
\tableofcontents
\maketitle
\section{Introduction}             %
\label{sec:introduction}           %
It is a basic idea to look at closed forms defined on the total space of a fibration 
 as a simple model for an abstract field theory. Indeed, a well known fact in symplectic geometry is that if a closed $2$-form, defined on the total space of
 a fibration $M\to B$, restricts to a symplectic form on the fibres, then it admits a symplectic connection,  \emph{i.e.} a connection whose parallel transport preserves the fibred symplectic form.
   In particular, up to completeness issues, the fibres can be identified as symplectic manifolds and the fibration can be regarded as a \emph{symplectic fibration}, \emph{i.e.} a locally trivial fibration with fibre type a symplectic manifold $(F,\omega_F)$.

In fact, the equations for such a form on $M$ to be closed have even stronger restrictions: they force the
 fibration to reduce  to a \emph{hamiltonian fibration}. Roughly speaking, there is a
  $\text{Sp}(F,\omega_F)$-principal bundle $P$  associated with the fibration, to be thought of as the
   \emph{symplectic frame} bundle. It turns out that $P$ admits a reduction to the subgroup
    $\text{Ham}(F,\omega_F)$ of \emph{hamiltonian} symplectomorphisms (say, up to a cover \cite{MD1}).
      Let us emphasize here that each of these facts has its infinitesimal counterpart.

From these simple observations, we obtain a universal way to build such fibrations: one shall consider associated bundles
 $M:=P\!\times_G\! F$ where $P$ is a principal bundle with a fixed connection, and where the structure group $G$ acts
  on the fibre type $F$ in a hamiltonian way.

This construction has many applications. In classical mechanics, the procedure describes the  reduction of canonical bundles, and is closely related with the coadjoint orbit hierarchy.  Also, Weinstein \cite{Wein3} recovered in this manner the minimal coupling of Sternberg \cite{Stern}, which describes the dynamics of a classical
    particle in a Yang-Mills field. In fact, many classical Yang-Mills theories can be recovered from this simple procedure.
     From this perspective, note that it is the form on the total space that will play the role of a "field".

A fundamental problem closely related to these questions is to exhibit obstructions to the existence of a closed extension. Namely,
 given closed forms varying smoothly on the fibres, a closed extension is a closed form defined on the total space that restricts to the forms prescribed fibre-wise. 
  Here, the difficulty is to ensure the closedness of this form. Heuristically, such obstructions
   shall be responsible for "anomalies" in physical situations. Notice that from both the mathematical and the  physical point of view,
    this happens when the topology of the total space does not allow the presence of a specific geometric structure (in our case, the closed $2$-form).

The aim of this work is to look at similar questions in the case of $3$-forms. One of the motivations is an attempt to obtain the simplest possible model
 for higher Yang-Mills theories.

The paper is organised as follows: first we need an infinitesimal approach to the Leray-Serre spectral sequence. For this, we will use an Ehresmann connection in order to decompose de Rham differential of the total space into a sum $\d=\delta_{0,1}+\delta_{1,0}+\delta_{2,-1}$.
 Here $\delta_{i,j}$ are differential operators of bi-degree $(i,j)$ with respect to a decomposition $\Omega^k(M)=\oplus\ \Omega^p(B,\Omega^q(\Ver))$
 induced by the connection, where $\Ver$ denotes the vertical bundle. This will be the basic tool of this paper and, as we shall explain, will help understanding
 the closedness equations as symmetries of structures defined fibre-wise. Note that in this work, the notion of "coupling" essentially refers to how the different
 components of a form assemble in order to define a closed one. 

The second section is devoted to the study of closed $2$-forms. Using the above decomposition, we will first review the classical theory of symplectic fibrations.
 Our approach here is designed to carry easily to both the study of presymplectic fibrations, and the generalization for $3$-forms we are aiming at.
  At this stage, we shall be able to point out obstructions in the third De Rham cohomology of the base.   Working in local coordinates,
   we will recover a $2$-connection with values in a crossed module and; we will explain how such phenomenons appear when dealing with \emph{non equivariant} moment maps.

We will then move to the non-symplectic case. In that situation, there is no reason for the existence of a connection whose parallel transport preserves the
 forms fiber-wise. However, when a closed extension exists we shall see that, though the forms may not be preserved, the \emph{prequantization} structure they induce on the fibers is. This will lead to a re-reading of the closedness equations for which the $2$-connection picture still makes sense.

In the last section, we will see how these constructions generalize to the case $3$-forms. First we will focus on at a basic notion of $2$-plectic fibration, then look at arbitrary closed $3$-forms. In that case, it is the structure of exact Courant algebroid defined fibre wise which is preserved, rather than the fibred $3$-form itself. Here we shall recover, partially but in a striking way, the analogue of a $3$-connection with values in a $2$-crossed module.

\section{The Leray-Serre spectral sequence}\label{Sec:Leray:Serre}
For a submersion $p:M\to B$, the Leray-Serre spectral sequence is induced by a filtration  
$${C^k=F_0 C^{k}\supset \dots\supset F_qC^{k}\supset F_{q+1}C^{k}\supset\dots\supset F_{k+1}C^{k}=\{0\}}$$
 defined on the complex of differential forms  by:
\begin{equation}\label{filtration}
F_q C^k:=\bigl\{\alpha\in\Omega^k(M)\ \bigl|\ i_v\alpha=0\quad \forall v\in\Gamma(\Lambda^{k-q+1}\Ver)\bigr\}.
 \end{equation}
Here $\Ver:=\ker \d p$ denotes the vertical bundle. In practice, this allows to break down cohomological computations into several
 steps by making finer and finer truncations. This construction is well known, nevertheless it remains a rather abstract process from a general perspective, and,
 depending on technical assumptions, there may be various approaches leading to slightly different interpretations (see the comments in \cite[Th.1]{GLWS}).

In this work, we want a first approach that will emphasize the infinitesimal point of view, this in order to keep track of the differential equations.
 For this purpose, the use of connections turns out to be very convenient as we explain below.

Yet for our construction to hold \emph{stricto sensu}, we need to assume that we are dealing with a \emph{complete} fibration. In fact, by \emph{fibration},
 we mean a smooth submersion that admits a complete Ehresmann connection. Note however that the decomposition \eqref{eq:deltadecomposition} makes sense in a much broader context,
  as well as the filtration and the spectral sequence, provided one deals with abstract quotients.

Recall briefly that an Ehresmann connection \cite{Ehr} is defined as a smooth splitting  $TM=\Ver\oplus \Hor$. We shall denote $h:\X(B)\to \Gamma(\Hor)$
 the horizontal lift, where $h(u)$ is defined as the unique section of $\Hor$ that $p_*$-projects to $u$, and $C\in\Omega^2(B,\Gamma(\Ver))$
 the curvature:
$$C(u,v):=\bigl[h(u),h(v)\bigr]-h\bigl([u,v]\bigr),\quad  u,v\in\X(B).$$
Given a connection, we can decompose $k$-forms on $M$ into several components, involving vertical and horizontal factors.
 More precisely, one may first consider the dual decomposition $T^*M=\Ver^0\oplus \Hor^0$ and make identifications $\Hor^0\simeq \Ver^*$ and $\Ver^0\simeq \Hor^*$.
 Then using the horizontal lift, we get an isomorphism $\Gamma(\Hor^*)\simeq \Omega^1(B)\otimes C^\infty(M)$ given by $\alpha\mapsto\langle\alpha, h(-)\rangle$. This
 identification naturally extends to the space of differential forms. We  obtain an isomorphism:
\begin{equation}\label{eq:filtration:forms}\Omega^k(M)\simeq \bigoplus_{p+q=k}\Omega^{p,q},
\end{equation}
where $\Omega^{{p,q}}:=\Omega^p(B)\otimes\Omega^q(\Ver)$. More precisely, the $(p,q)$-component
 $\Theta^{_{(p,q)}}\in\Omega^{p,q}$ of a $k$-form $\Theta\in\Omega^k(M)$ is given by:
\begin{equation}\label{splitkforms}\big\langle\Theta^{_{(p,q)}}(v),w\big\rangle  
                                        =\big\langle\Theta,h(v){_\wedge} w\big\rangle, \quad v\in\Gamma(\wedge^pTB),\ w\in \Gamma(\wedge^q \Ver).
\end{equation}
It will be meaningful for us to think of an element in $\Omega^{p,q}$ as a $p$-form on $B$ with values in $q$-forms on the fibres,
 so we shall write $\Omega^p(B,\Omega^q(\Ver))$ rather than $\Omega^{p,q}$. Note that under the identification
 \eqref{eq:filtration:forms} the filtration \eqref{filtration} reads:
$$F_q C^{k}=\bigoplus_{q\leq p\leq k}\Omega^p(B, \Omega^{k-p}(\Ver)). $$
The decomposition \eqref{eq:filtration:forms} allows to make rather explicit computations, as the de Rham differential on $M$ decomposes accordingly into a sum:
\begin{equation*}\label{eq:deltadecomposition}\d=\delta_{0,1}+\delta_{1,0}+\delta_{2,-1}.\end{equation*}
Detailed computations for this fact may be found in \cite{Br}. Here $\delta_{i,j}$ has bi-degree $(i,j)$ and can be described as follows:
\begin{itemize}
\item first recall that there is a well defined \emph{vertical} De Rham operator $\d_V:\Omega^k(\Ver)\to \Omega^{k+1}(\Ver)$ given by the usual formula:
\begin{multline*}
 (\d_V\eta)(v_0\dots v_n)=\sum_{i=0\dots p}(-1)^{i}\Lie_{v_i}\eta(v_1\dots\widehat{v_i}\dots v_p)\\
+\sum_{i<j}(-1)^{i+j}\eta([v_i,v_j],v_1\dots\widehat{v_i}\dots\widehat{v_j}\dots v_p),
\end{multline*}
where $\eta\in\Omega^k(\Ver)$ and $v_i\in \Gamma(\Ver)$. Then $\delta_{0,1}^{_{(p,q)}}:\Omega^{p,q}\to \Omega^{p,q+1}$ naturally extends the vertical
 De Rham differential by tensor product: $\delta^{_{(p)}}_{0,1}=(-1)^p\d_V\otimes \id$.  As we think of $\Theta$ as a $p$-form with values in $q$-forms,
 we shall write $$\delta_{0,1}^{_{(p,q)}}\Theta^{p,q}=(-1)^p \d_{V}\circ\Theta^{p,q}.$$
\item the operator $\delta_{1,0}^{_{(p,q)}}:\Omega^{p,q}\to \Omega^{p+1,q}$ is the covariant derivative $\partialH$
       associated to the connection. Recall that for an Ehresmann connection, this is an operator
 $\partialH:\Omega^p(B,\Omega^q(\Ver))\to\Omega^{p+1}(B,\Omega^q(\Ver))$ given by the following formula, where $v_1\dots v_n\in\X(B)$:
\begin{multline*}\label{eq:covariant:derivative}
(\partialH \Lambda)(v_0\dots v_p)
=\sum_{i=0\dots p}(-1)^{i}\Lie_{h(v_i)} \bigl(\Lambda(v_1\dots\widehat{v_i}\dots v_p)\bigr)\\
+\sum_{i<j}(-1)^{i+j}\Lambda([v_i,v_j],v_1\dots\widehat{v_i}\dots\widehat{v_j}\dots v_p),
\end{multline*}
Note in particular that for $\Lambda\in \Omega^k(\Ver)$ we have $(\partialH \Lambda)(u)=\Lie_{h(u)} \Lambda$ so  that $\partialH$ measures the parallel
 transport of the connection. Recall also that $\partialH\circ\partialH\neq 0$ unless the curvature vanishes.
 \item the operator $\delta_{2,-1}^{_{(p,q)}}:\Omega^{p,q}\to \Omega^{p+2,q-1}$ is obtained by anti-symmetrizing the contraction with the curvature. More precisely:
\begin{multline*}(\delta_{2,-1}^{_{(p,q)}}\Theta^{_{(p,q)}})(v_0\dots v_{p+1})\\
=(-1)^{p}\sum_{i<j}(-1)^{i+j} i_{C(v_i,v_j)}\Theta^{_{(p,q)}}(v_0\dots \widehat{v_i}\dots\widehat{v_j}\dots v_{p+1}).\end{multline*}
\end{itemize}

In the sequel, we will write $\partialV, \partialH$ and $\partialC $ as  generic notations for respectively
 $\delta_{1,0}^{_{(p,q)}},\delta_{0,1}^{_{(p,q)}}$ and $\delta_{2,-1}^{_{(p,q)}}$ in any degree. A useful diagram to picture the situation
 is the following:
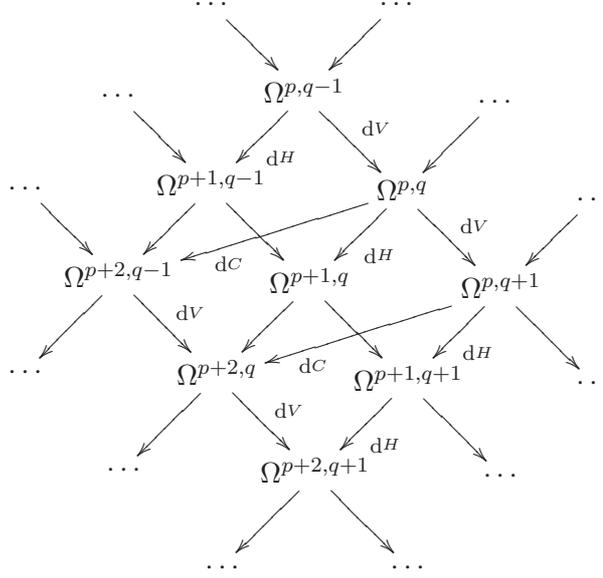
\vspace{-1cm}
$$\xymatrix@dl@R=30pt@C=27pt{  & \cdots                                                  &\cdots             & \cdots         \\
 \cdots\ar[r] &\Omega^{p,q+1}\ar[r]^{\partialH} \ar[u]   \ar[drr]^>>>>>>{\partialC}       & \Omega^{p+1,q+1} \ar[r]^{\partialH}\ar[u]&\Omega^{p+2,q+1}\ar[r]\ar[u] &\cdots \\
  \cdots\ar[r]&\Omega^{p,q} \ar[r]^{\partialH} \ar[u]^{\partialV}  \ar[drr]^>>>>>>{\partialC}         &\Omega^{p+1,q} \ar[u] \ar[r]&\Omega^{p+2,q}\ar[r]\ar[u]^{\partialV} &\cdots\\
 \cdots\ar[r] &\Omega^{p,q-1}  \ar[r]^{\partialH} \ar[u]^{\partialV}       & \Omega^{p+1,q-1} \ar[r]\ar[u]&\Omega^{p+2,q-1} \ar[r]\ar[u]^{\partialV} &\cdots \\
              & \cdots \ar[u]                                           &\cdots\ar[u]&\cdots\ar[u]}$$ 

By using the bi-grading, the equation $\d\circ\d=0$ easily yields the following commutations relations. These can be read as $\partialH$
 inducing a representation up to homotopy of $\X(B)$ on the vertical De Rham complex:
\begin{equation}\label{eq:commutation:relations}\left\{\begin{array}{lcr}
 \partialV\circ \partialV                              &=&0,\\
 \partialV\circ \partialH+\partialH\circ \partialV    \hspace{53pt}&=&0,\\
  \partialV\circ \partialC \hspace{0.5pt}+\partialH\circ\partialH +\partialC\circ \partialV   &=&0,\\
 \hspace{47pt}\partialH \circ \partialC + \partialC\circ \partialH&=&0,\\
 \hspace{93pt}\partialC\circ \partialC \ &=&0.
               \end{array}\right.
\end{equation}
Let us now try to portray the general situation by working out the usual procedure of a spectral sequence. Firstly, we simply obtain the $E_0$ term as
 the complex $E_0^{p,q}=\Omega^{p,q}$ with $\partialV$ as coboundary. Thus cochains in the $E_1$ term
 get identified with forms in $B$ with values in the vertical de Rham cohomology and we have $E_1^{p,q}=\Omega^p(B,H^q(\Ver))$. In the sequel,
 we  will denote $[\omega]_1\in E_1^{p,q}$ the class of a coboundary $\omega\in \Omega^p(B,\Omega^q(\Ver))$. Notice that one can really think
 of $H^\bullet(\Ver)$ as a vector bundle over $B$ as, by the completeness assumption $H^\bullet(\Ver)$ is easily seen to be \emph{locally} trivial over $B$.

By construction, the differential operator on $E_1$ is induced by the connection. In fact, the second equality in
 \eqref{eq:commutation:relations} clearly states that $\partialH$ factors out to a connection on  $H^\bullet(\Ver)$.
 Furthermore, it immediately follows from the third equality in \eqref{eq:commutation:relations} that the induced connection is \emph{flat}, so that $H^\bullet(\Ver)$
 behaves like a \emph{flat vector bundle} over $B$. In terms of Lie algebroids, this means that $H^\bullet(\Ver)$ is a representation of $TB$. Then
 the $E_2$ term is the associated cohomology. For this reason, we will denote $E_2^{p,q}$ by $H^p(B,H^q(\Ver))$ and $[\omega]_2\in E_2^{p,q}$
 the class of a coboundary $[\omega]_1\in E_1^{p,q}$. In fact, one can easily check that this representation is canonical, in the sense that
 is independent of the choice of the connection.

Of course, going further in the spectral sequence requires more effort. Instead, we will introduce geometric structures in the picture in order 
 to get a better insight. For now on, let us give, as a first application, a proof of a classical Theorem due to Thurston:
\begin{theorem}\label{thm:closed:extension}
 Let $M\xrightarrow{p} B$ be a fibration, and $\Theta_V$ a smooth family of closed $k$-forms on the fibres. Then $\Theta_V$ admits a closed extension if and only if there exists a cohomology class
 $[\Lambda]\in H^k(M)$ whose restriction to the fibres is $[\Theta_V]_1\in H^k(\Ver)$.
\end{theorem}
\begin{proof}
By the assumptions we consider $\Theta_V\in\Omega^k(\Ver)$ with $\partialV\Theta_V=0.$

Assume that $\Theta_V$ admits a closed extension, \emph{i.e.} that there exists $\Lambda\in\Omega^k(M)$ such that  $\d\Lambda=0$ and $\Theta_V=\Lambda_V$,
 where $\Lambda_V:=\Lambda|_{\wedge^k\Ver}$. Then obviously $[\Lambda_{V}]_1=[\Theta_V]_1$ since $\Theta_V=\Lambda_V$.

Conversely, assume that there exists a closed form $\Lambda\in\Omega^k(M)$ such that $[\Lambda_V]_1=[\Theta_V]_1$. This precisely means that
 there exists $\alpha_V\in \Omega^{k-1}(\Ver)$ such that $\Theta_V=\Lambda_{V}+\partialV\alpha_V.$ We use \eqref{eq:filtration:forms} in order to
 identify $\alpha_V$ with an element in $\Omega^{k-1}(M)$ with one single component in $\Omega^{k-1}(\Ver)$, equivalently an element in $\Omega^{k-1}(\Hor^0)$. Then we compute that
 $\d\alpha_V=\partialV\alpha_V\oplus\ \partialH\alpha_V\oplus\partialC\alpha_V\oplus 0\oplus\dots\oplus 0.$ Therefore
$\Lambda+\d \alpha_V=\Lambda_V+\partialV \alpha_V\oplus \dots=\Theta_V\oplus\cdots$

We conclude that $\Lambda+\d \alpha_V$ is indeed a closed extension of $\Theta_V$.\end{proof}

\section{Closed 2-forms on a fibration}\label{sec:omega:decomposition}
For a $2$-form $\omega\in\Omega^2(M)$ the decomposition \eqref{eq:filtration:forms} reads as follows:
\begin{equation}\label{eq:omega:decomposition}
 \omega=\omega_V\oplus\alpha_H\oplus\omega_H.
\end{equation}
More explicitly, in this expression:
\begin{itemize}
 \item $\omega_V$ is a vertical $2$-form, \emph{i.e.} a $2$-form on the fibres: $\omega_V\in \Omega^2(\Ver)$
 \item $\alpha$ is a form with values in vertical $1$-forms: $\alpha\in\Omega^1(B,\Omega^1(\Ver))$,
 \item $\omega_{H}$ is a $2$-form with values in functions: $\omega_{H}\in\Omega^2(B,C^\infty(M))$.
\end{itemize}
Then the corresponding decomposition for $\d\omega$ yields the equations for $\omega$ to be closed:
\begin{eqnarray}
 \partialV\omega_V\hspace{84pt}&=&0,\label{eq:closed2:1}\\
 \partialH\omega_V+\partialV\alpha_H\hspace{42pt}&=&0,\label{eq:closed2:2}\\
 \partialC \omega_V+\partialH\alpha_H+\partialV\omega_H\hspace{3pt}&=&0,\label{eq:closed2:3}\\
 \partialC\alpha_H+\partialH\omega_H\hspace{2pt}&=&0\label{eq:closed2:4}.
\end{eqnarray}
In particular, the first equations above are easily interpreted:
\begin{itemize}
 \item $\omega_V$ is vertically closed, thus it induces a cohomology class in the vertical De Rham cohomology $[\omega_V]_1\in H^2(\Ver)$ seen as a vector bundle over $B$;
 \item the covariant derivative of $\omega_V$ takes values in exact forms for which $\alpha_H$ provides primitives.
  Here we see that $[\omega_V]_1$ defines a parallel section of $H^\bullet(\Ver)$ as a flat vector bundle.
  More precisely, for any complete vector field $u\in\X(B)$ we obtain by integrating \eqref{eq:closed2:2} the following:
\begin{equation}\label{eq:integration:2form}
 \bigl(\phi^{h(u)}_t\bigr)_*\omega_V=\omega_V+\d_V\Bigl(\,\int_0^t (\phi^{h(u)}_s)_*\alpha_H(u) \d s\Bigr).
\end{equation}
\end{itemize}
The other equations seem more intricate for an interpretation, we will postpone this question to the Section \ref{sec:gaugegeneral2forms},
 and focus on the classical treatment for symplectic fibrations. 
\subsection{Symplectic Fibrations}\label{sec:symplectic:fibrations}
In this section, we will review well-known facts concerning symplectic fibrations, formerly studied in \cite{GLWS}, \cite{Wein2}.
 None of the results here is original, except maybe for the Proposition \ref{prop:obstr:omegaclosed2}. Our main concern is to formalize some ideas
  from \cite{GLS} so that the generalization to $3$-forms will be straightforward. We refer the interested reader to \cite{LD}, \cite{MD1} for
   a topological approach, see also \cite{Br1}, \cite{Ts} where such questions are treated using gerbes.

In order to study symplectic fibrations, one approach is to start with a Cech cocycle with values in the group of symplectic diffeomorphisms $\Symp(F,\omega_F)$ for
 a given model $(F,\omega_F)$. There is however a slightly different approach, which is by dealing only with vertical forms and connections. Note that by doing so, we get rid of a specific model $(F,\omega_F)$ from the very beginning.

Consider a fibration $M\to B$ together with a symplectic form varying smoothly on the fibres. More precisely we consider a vertically closed
 non-degenerate $2$-form, namely $\omega_V\in\Omega^2(\Ver)$ satisfying:
\begin{align*}
 \partialV\!\omega_V&=0,\ \ \\
  \omega_V^{\wedge 2n}&\neq 0,
\end{align*}
where $n$ denotes the rank of $\Ver$. In that situation, a \textbf{symplectic connection} is a connection whose parallel transport preserves $\omega_V$:
\begin{equation}\label{eq:sp:connection}
 \partialH\!\,\omega_V=0.
\end{equation}
\begin{definition}A \textbf{symplectic fibration} is a couple $(M\to B,\omega_V)$ as above, such that $\omega_V$ admits a symplectic connection.
\end{definition}
\begin{remark} 
In this definition, we assume implicitly that the connection is complete. Thus one can identify the fibres with each other
 as symplectic manifolds by using the parallel transport, and the condition \eqref{eq:sp:connection} is just the infinitesimal counterpart
 for a fibration to be locally trivial.
\end{remark}
First we look at the case where $\omega_V$ is the restriction of a $2$-form defined on the total space $M$. The following characterization
 will useful for later purposes, the proof is completely tautological and will be left to the reader:
\begin{proposition}\label{prop:fibre:non:degeneracy:2form}
 Let $\omega\in\Omega^2(M)$ be a $2$-form defined on the total space of a fibration $M\xrightarrow{p} B$.
 Then the following assertions are equivalent:
\begin{enumerate}[i)]
\item the restriction of $\omega$ to the fibres of $p$ is non-degenerate, 
\item $\Ver^{\perp\omega}\cap\Ver=\{0\}$,
\item $\Graph(\omega)\cap \Ver\oplus\Ver^0=\{0\}.$
\end{enumerate}
\end{proposition}
Here, $\Ver^0\subset T^*M$ denotes the annihilator of the vertical bundle.
\begin{definition}
 A $2$-form $\omega\in\Omega^2(M)$ is said to be \textbf{fibre non-degenerate} if one of the conditions
 above is satisfied.
\end{definition}
When a fibre non-degenerate $2$-form is closed, we have the following fundamental observation due to Gotay \emph{et al.} \cite{GLWS}:
\begin{proposition}\label{prop:induced:symp:connection}
 Let $\omega\in\Omega^2(M)$ be a closed, fibre non-degenerate  $2$-form defined on $M\to B$. 
 Then $(M\to B,\omega_V)$ admits a symplectic connection with hamiltonian curvature.
\end{proposition}
\begin{proof}
By letting $\Hor:=\Ver^{\perp\omega}$ we obtain, according to the proposition \ref{prop:fibre:non:degeneracy:2form}, a vector bundle complementary to $\Ver$ in $TM$,
 therefore a well defined Ehresmann connection. For this choice of horizontal subspace, the component $\alpha_H$ in the decomposition \eqref{eq:omega:decomposition}
 vanishes because of \eqref{splitkforms}.

Then it suffices to set $\alpha_H=0$ in \eqref{eq:closed2:2} and \eqref{eq:closed2:3} to obtain the following:
\begin{align*}
\forall u,v\in\X(B) \quad\Lie_{h(u)}\omega_V=&\ 0,\\
i_{C(u,v)}\omega_V=&-\partialV\omega_H(u,v).
\end{align*}
The first equation precisely states that the connection is symplectic, and the second
that $\omega_H$ prescribes hamiltonian functions for the curvature elements. Note that the closedness condition \eqref{eq:closed2:4} implies
 furthermore that $\omega_H$ shall be $\partialH$-closed.
\end{proof}

Conversely, starting with a symplectic fibration, one may try to refine Thm. \ref{thm:closed:extension} in order to decide whether $\omega_V$
 admits a closed extension.

Once fixed a symplectic connection, one can choose $\alpha_H=0$ to get equation
 \eqref{eq:closed2:2} satisfied. Notice however that we have some freedom in the choice of $\alpha_H$,
 as this equation only requires that it takes values in \emph{closed} $1$-forms, \emph{i.e.}
 that $\alpha_H\in\Omega^1(B,\Omega^1_{cl}(\Ver))$. In fact, this choice may be seen alternatively as varying the connection:

\begin{proposition}\label{symp:connections:affine}
 For any symplectic fibration $(M \to B,\omega_V)$, the space of symplectic connections is 
 an affine space modelled on $\Omega^1(B,\Omega^1_{cl}(\Ver))$. 
\end{proposition}
\begin{proof}
The horizontal lifts $h,h'$ associated to different symplectic connections differ by $h'-h=\Delta$ where
 $\Delta\in\Omega^1(B,\Gamma_{\omega_V}(\Ver))$. Here $\Gamma_{\omega_V}(\Ver)$ denotes the space of vertical vector fields
 that preserve $\omega_V$, \emph{i.e.} satisfying $\Lie_{X_V}\omega_V=0.$ Then, since $\omega_V$ is symplectic it establishes an identification:
$\Omega^1(B,\Gamma_{\omega_V}(\Ver))\simeq\Omega^1(B,\Omega^1_{cl}(\Ver))$
obtained by composition with $i_{_-}{\omega_V}:\Gamma_{\omega_V}(\Ver)\xrightarrow{_\sim}\Omega^1_{cl}(\Ver)$.
\end{proof}

\begin{proposition}\label{prop:obstr:omegaclosed1}
Let $(M,\omega_V)$ be a symplectic fibration with connected fibres. Then there is a class $[\partialC\omega_V]_2\in H^2(B,H^1(\Ver))$ which is an obstruction
 to the existence of a symplectic connection with hamiltonian curvature.
\end{proposition}
\begin{proof}
 First we use the commutation relations \eqref{eq:commutation:relations} to check that:
 \begin{equation*}\partialV\circ\partialC\omega_V=-\partialH\circ\partialH\omega_V-\partialC\circ\partialV\omega_V=0.
 \end{equation*}
 This ensures that $[\partialC\omega_V]_1$ is a well defined element in $\Omega^2(B,H^1(\Ver))$.

 Using \eqref{eq:commutation:relations} again, we obtain
$\partialH\circ\partialC \omega_V=-\partialC\circ \partialH\omega_V=0$, therefore $[\partialC\omega_V]_1\in\Omega^2(B,H^1(\Ver))$ is a cocycle and the corresponding class
 $[\partialC\omega_V]_2\in H^2(B,H^1(\Ver))$ is well defined.

Assume now that this class vanishes; this precisely means that there exists $\alpha\in\Omega^1(B,\Omega^1_{cl}(\Ver))$ such that
\begin{equation*}[\partialC\omega_V]_2=\partialH[-\alpha_H]_2.
 \end{equation*}
The above equality lies in $\Omega^2(B,H^1(\Ver))$ so $\partialC\omega_V $ and $\partialH \alpha_H$ need only to
 coincide up to some $\partialV\omega_H$, where $\omega_H\in\Omega^2(B,C^\infty(M))$. In conclusion, the vanishing of $[\partialC\omega_V]_2$
 is equivalent to the existence of $\alpha_H$ and $\omega_H$ satisfying $\partialC \omega_V=-\partialH\alpha_H-\partialV\omega_H$ and $\partialV\alpha_H=0$.

Consider now the connection $h':=h+\Delta$ corresponding to $\alpha_H$ as in the Proposition \ref{symp:connections:affine} by letting $\alpha_H:=i_{\Delta}\omega_V$.
 Notice that $h'$ is symplectic because $\alpha_H$ takes values in closed $1$-forms. Moreover $h'$ has curvature $C'=C+\partialH\Delta+[\Delta_\wedge\Delta]$
 where $[\Delta_\wedge\Delta](X,Y):=[\Delta(X),\Delta(Y)]$. Then the following computation shows that $C'$ takes values in hamiltonian vector fields,
 with hamiltonian prescribed by $\omega_H'=\omega_H+\omega_V(\Delta_\wedge\Delta)$:
\begin{align*}\partialCp\omega_V=&\partialC{\omega_V}+\partialH\alpha_H+[{\alpha_H}_\wedge\alpha_H]_{\Ver^*}\\
                                  =&-\partialV\omega_H + [{\alpha_H}_\wedge\alpha_H]_{\Ver^*}\\                                 
                                  =&-\partialV(\omega_H+\omega_V(\Delta_\wedge\Delta)).
\end{align*}
Here $[{\alpha_H}_\wedge \alpha_H]_{\Ver^*}$ is defined using the Poisson bracket induced on the space of vertical $1$-forms induced by $\omega_V$, and we used
 the general fact that the bracket of closed $1$-forms is exact. More precisely for any $u,v\in \X(B)$, we have
 $$[{\alpha_H}_\wedge\alpha_H]_{\Ver^*}(u,v):=[\alpha_H(u),\alpha_H(v)]_{\Ver^*}=\partialV \omega_V(\Delta(u),\Delta(v)).$$\end{proof}
\begin{proposition}\label{prop:obstr:omegaclosed2}
Consider  a symplectic fibration $(M\xrightarrow{p}B,\omega_V)$ with connected fibres, and assume that it admits a connection with hamiltonian curvature.
 Then there exists a $2$-form $\omega\in\Omega^2(M)$ extending $\omega_V$ such that $\d \omega=p^*\phi,$ where $\phi$ is a closed $3$-form on $B$. Moreover,
  if $[\phi]\in H^3(B)$ vanishes then $\omega_V$ admits a closed extension.
\end{proposition}
\begin{proof}
Denote by $\omega_H\in\Omega^2(B,C^\infty(M))$ the form specifying hamiltonians for the curvature.
By the assumptions, if we let $\omega=\omega_V\oplus 0 \oplus \omega_H$ then all the equations ensuring $\d \omega=0$ are satisfied except for \eqref{eq:closed2:4}.
Namely, we have:
\begin{align*}
 \partialV \omega_V \hspace{38pt}=&0,\\
 \partialH\omega_V\hspace{38pt}=&0,\\
 \partialC\omega_V+\partialV\omega_H=&0.
\end{align*}
By using successively the commutation relations \eqref{eq:commutation:relations} and the above conditions, we obtain
$\partialV\circ\partialH\omega_H
=-\partialH\circ\partialV\omega_H
=\partialH\circ\partialC\omega_V
=\partialC\circ\partialH\omega_V
=0.$ We conclude that $\partialH\omega_H$ takes values in fibre-wise constant functions, so it defines a $3$-form $\phi$ on $B$
 such that $\d\omega=\partialH\omega_H=p^*\phi$.
As the following  computation shows, $\phi$ is necessarily closed, this is a purely combinatorial fact,
 as it follows directly from the definition of $\partialC$ and the anti-symmetricity of $\omega_V$. Indeed we have:
\begin{align*}
&\ (\partialH \circ\partialH\omega_H)(u_0,u_1,u_2,u_3)\\
=&\ (\partialC\circ\partialV\omega_H)(u_0,u_1,u_2,u_3)\\
=&\ \langle C(u_0,u_1),\partialV\omega_H(u_2,u_3)\rangle -\langle C(u_0,u_2),\partialV\omega_H(u_1,u_3)\rangle\\
 &          \ \quad\quad+\langle C(u_0,u_3),\partialV\omega_H(u_1,u_2)\rangle+\langle C(u_1,u_2),\partialV\omega_H(u_0,u_3)\rangle\\
 &\ \quad\quad\quad\quad -\langle C(u_1,u_3),\partialV\omega_H(u_0,u_2)\rangle +\langle C(u_2,u_3),\partialV\omega_H(u_0,u_1)\rangle\\
=&\ \omega_V( C(u_0,u_1),C(u_2,u_3)) -\omega_V( C(u_0,u_2),C(u_1,u_3))\\
 &\ \quad\quad     +\omega_V(C(u_0,u_3),C(u_1,u_2))+\omega_V( C(u_1,u_2),C(u_0,u_3)) \\
&\   \quad\quad\quad\quad -\omega_V( C(u_1,u_3),C(u_0,u_2))+\omega_V(C(u_2,u_3),C(u_0,u_1))=0.
\end{align*}

Assume now that the class $[\phi]=[\partialH\omega_H]\in H^3(B)$ vanishes. Then $\phi=\d \omega_B$ where $\omega_B\in\Omega^2(B)$ and
 we can replace $\omega_H$ by $\omega_H+p^*\omega_B$ to get all equations \eqref{eq:closed2:1}-\eqref{eq:closed2:4} satisfied.\end{proof}

Notice that it requires rather strong topological constraints on the fibration for the last obstruction class not to vanish, as $[\phi]$ shall define an element in the kernel of the map $p^*:H^3(B)\to H^3(M)$ induced in cohomology.

\begin{remark}\label{rem:final:obstruction:2form}
The Propositions \ref{prop:obstr:omegaclosed1} and  \ref{prop:obstr:omegaclosed2} are rather basic results, that has subtle refinements.
 For instance, when $p$ has $1$-connected fibres, the first obstruction (Prop. \ref{prop:obstr:omegaclosed1})
 vanishes of course. More surprisingly, if we assume moreover that the fibre type $F$ is compact then the second obstruction
 (Prop. \ref{prop:obstr:omegaclosed2}) also vanishes, this was proved by Guillemin-Lerman-Sternberg using an averaging argument \cite[Thm 1.4.1]{GLS}.
 \comment{see also  \cite[Lem 1.6]{MD1}.}

If we do not assume the fibres to be simply connected, the vanishing of the first obstruction is a difficult problem in practice;
this was studied by McDuff-Lalonde from a topological point of view in \cite{LD} for compact fibres, see also \cite{MD1}.

 In the context of gerbes, Brylinsky \cite{Br1} looked at such problems (see also Tsemo \cite{Ts}). \comment{precisely \cite[Thm 2.6.6]{Ts}}
Roughly speaking, gerbes get involved when trying to glue principal bundles defined over the fibre type. Assuming that the fibred $2$-form is integral, Brylinky obtained in this way
 a class in $H^3(B,\mathbb{Z})$.
\comment{Tsemo gives an obstruction as an element in $H^2(B,H^1(F,\mathbb{R})/H^1(F,P_{\omega_F}))$}
\comment{Yet, it is not sooo clear to relate precisely this obstruction with Prop. \ref{prop:obstr:omegaclosed1} and \ref{prop:obstr:omegaclosed2}

 Maybe should we think of a class in $H^2(B,H^1(F,\mathbb{R})/H^1(F,P_{\omega_F}))$ being somehow represented by 
an element in $H^3(B, something)$ which is 'integral' in the sense that it integrates along cycles always lead to elements in $H^1(F,P_{\omega_F}))$,
 on the lines of \cite[p. 321]{LD}  

(In fact, here I would like to reproduce the argument that $H^2(B,\mathbb{R}/\mathbb{Z})\simeq H^3(B,\mathbb{Z})$ whose elements can be represented by integral
 De Rham classes, \emph{i.e.} a usual De Rham class $\phi\in H^3(B,\mathbb{R})$ so that $\int_C\phi\in\mathbb{Z}$ for all cycle)}
\end{remark}

\subsection{Weakly Hamiltonian fibrations as 2-connections}\label{sec:local:symplectic}
As we just pointed out, the Proposition \ref{prop:obstr:omegaclosed2} has integral counterparts in gerbe theory. From our approach, one may expect some sort of an higher analogue of  connection. It is indeed the case, but we need first to introduce some 
 terminology.
\begin{definition}
 A \textbf{weakly hamiltonian fibration} is a symplectic fibration $(M\to B,\omega_V)$ together with a connection that has Hamiltonian curvature.
 Such a connection will be said to be \textbf{hamiltonian} if the $2$-form $\omega_H$ prescribing the hamiltonian functions can be chosen so that $\partialH\omega_H=0$.
\end{definition}
Let us take a closer look at weakly hamiltonian fibrations. By the definition, the following conditions hold:
\begin{align}
0&=\partialH\omega_V,\label{eq:ham:curv1}\\
0&=\partialC\omega_V+\partialV\omega_H,\label{eq:ham:curv2}
\end{align}
where $\omega_H\in\Omega^2(B,C^\infty(M))$. Furthermore this setting implies, according to the proposition  \ref{prop:obstr:omegaclosed2}, the existence of a $3$-form $\phi\in\Omega^3(B)$ such that:
\begin{align}
\hspace{50pt}\partialH\omega_H&=p^*\phi\label{eq;ham:curv3}.
\end{align}
As we will explain in this section, the above equations correspond in fact to those of a so-called $2$-connection with vanishing fake curvature. Before this, we shall first recall some basic notions concerning crossed modules.

\subsubsection{Connections with values in a crossed module}

\begin{definition} A crossed module of groups consists into a couple of groups $G$ and $H$, together with a morphism $t:H \to G$ and a left action $G\to \text{Aut}(H)$ of $G$ on $H$, such that the following conditions are satisfied:
\begin{align*}
t(g \cdot h)&= g t(h) g^{-1}, \\
t(h)\cdot h'&=h h' h^{-1},
\end{align*}
for any $h,h'\in H,\ g\in G$.
\end{definition}

Crossed modules are categorified versions of groups, as they encode a (strict) $2$-group, \emph{i.e.} a category internal to the category of groups \cite{BaezLauda}. The infinitesimal counterpart of a crossed module of groups is a crossed module of Lie algebras:
\begin{definition} A crossed module of Lie algebras is given by a
a pair of Lie algebras $\mathfrak{g}$ and $\mathfrak{h}$, together with a morphism $\d t: \mathfrak{h}\to \mathfrak{g}$, and a Lie algebra morphism $\triangleright:\mathfrak{g}\to \Der(\mathfrak{h})$ such that the following conditions are satisfied:
\begin{align*}
\d t ( \xi \triangleright \eta)&=[x, \d t (\eta)],\\
\d t (\eta) \triangleright \eta'&=[\eta,\eta'],
\end{align*}
for any $\eta,\eta'\in \mathfrak{h},\ \xi\in\mathfrak{g}$.
\end{definition}

As is the case for crossed modules of groups, crossed modules of Lie algebras encode a category internal to Lie algebras \cite{BaezCrans}. There is also a categorical version of principal bundles with structure a $2$-group, as well as connections for such bundles (see Breen-Messing \cite{BrMe} for more details. In this work though, we will only use basic such notions, see for instance the work of Baez-Huerta-Shreiber \cite{BaezYM}, \cite{BS1}, \cite{BH}).
 
Given a crossed module of Lie algebras, the simplest instance of a $2$-connection $(a,b)$ can be described locally by a $\g$-valued $1$-form $a$ and an $\h$-valued $2$-form $b$. The curvature of $(a,b)$ is then defined as a couple $(F,\phi)$, where: 
$$F:=\d t(b)+\d a+ a_\wedge a$$
 is called the \emph{fake curvature}, while the curvature $3$-form $\phi$ is given by:
$$\phi:=\d_a b=\d b+a_\wedge b.$$

 Such objects shall be understood as ``categorical connections'', the basic idea is that the holonomy shall be defined along surfaces rather than paths. 
In this context, the vanishing of the fake curvature is an important feature as it ensures the existence of a well defined holonomy,  to be understood as a $2$-functor defined on the path $2$-groupoid of the base manifold (see the work of Schreiber-Waldorf \cite{Schreiber-Waldorf1},\cite{Schreiber-Waldorf2} and Martins-Picken \cite{MartinsPicken1},\cite{MartinsPicken2}).

\subsubsection{Weakly hamiltonian fibrations as $2$-connections}

 As we shall see in this work, the closure equation for a $2$-form can be recovered in that language. In the case of a weakly hamiltonian fibration, one may
 proceed as follows. First assume that $B\simeq \mathcal{U}$ is an open set identified with $\mathbb{R}^n$, with coordinates $x=(x_1,\dots,x_n)$, we denote $\partial_i$ the corresponding vector fields. 

We need first to build a trivialization for which $\omega_V$ is independent of the basic variables. This is rather trivial but let us give more details, as these will be relevant later on. Since we assume the connection to be complete, we can use the holonomy along straight lines through $0$ (say, the trajectories of the radial vector field $r=\sum_ix_i\partial_i$ on $\mathbb{R}^n$) in order to identify $p^{-1}(\U)$ with $\simeq F\times \U$, where $F:=p^{-1}(\{0\})$. Since $(\partialH\omega_V)(r)=0$ we see that $\omega_V$,
 which is \emph{\`a priori} a $2$-form on $F$ with parameters in $\mathbb{R}^n$, turns out to be independent of $(x_1,\dots,x_n)$ and we obtain a trivialization with $\omega_V\simeq\omega_F$.  In this trivialization,  the horizontal lifts of the connection write:
$$h(\partial_i)\loc\partial_i+a_i,\text{ where }a_i:\U \to \X(F).$$
Furthermore, the curvature is locally given by:
\begin{align*}
 C(\partial_i,\partial_j)&:=   [h(\partial_i),h(\partial_j)]-h([\partial_i,\partial_j])\\
                         & \loc   [\partial_i+a_i, \partial_j+a_j ]\\
                         & =    \partial_i a_j-\partial_j a_i+[a_i,a_j].
\end{align*} A common notation for the last equality is simply $C\loc\d a+ a_\wedge a$, where $a:=\sum_i a_i\d x_i\in \Omega^1(U,\X(F))$
 is the local connection form.
Then the invariance condition  \eqref{eq:ham:curv1} of $\omega_V$ by the connection reads:
\begin{align*}
(\partialH\omega_F)(\partial_i)=0
\iff&(\partial_i+\Lie_{a_i})(\omega_F)=0\\
\iff& \Lie_{a_i} \omega_F=0.
\end{align*}
Here, we see that the local connection form takes values in \emph{symplectic} vector fields.
 Furthermore, taking local coefficients for $\omega_H\loc\sum_{i,j} b_{i,j}\d x_i{_\wedge}\d x_j$
 we observe that $b=\sum_{i,j}b_{i,j}\d x_i{_\wedge}\d x_j$ is a $2$-form with values in functions on $F$.

In order to recover a $2$-connection, this suggests to consider the following crossed module of Lie algebras:
\begin{equation*}\label{crossed:module:symplectic}C^{\infty}(F)\xrightarrow{\d t}\X(F)_{\omega_F}.
\end{equation*}
Here the crossed module structure is given as follows:
\begin{itemize}
 \item $C^\infty(F)$ is endowed with the Poisson bracket $\{\ ,\ \}$ induced by $\omega_F$,
 \item $\X(F)_{\omega_F}$ is the space of symplectic vector fields (\emph{i.e.} satisfying $\Lie_X\omega_F=0$) acting on $C^\infty(F)$ in the obvious way: $X\triangleright f:=\Lie_X f$, 
  \item $\d t(f):=\{f,\ \}$ or, equivalently $t(f)=\omega_F^{-1}(\d f)$ where $\omega_V\!:\!TF\!\to T^*F$, $X\mapsto i_X\omega_F$ denotes the contraction.
\end{itemize}
In other words, this crossed module is just the geometric variation for $\g\xrightarrow{\ad} \text{Aut}(\g)$ in the case where $\g=(C^\infty(F),\{\ ,\ \})$.
 With this in mind, we can now express the equation \eqref{eq:ham:curv1} locally as follows:
\begin{align*}
     \bigl(\partialC\omega_V+\partialV\omega_H\bigr)(\partial_i,\partial_j)=0 
\iff&  i_{C_{i,j}} \omega_F+ \d_V\omega_H^{i,j}=0\\
\iff&  C_{i,j}+(\omega_F^{-1})(\d_V\omega_H^{i,j})=0\\
\iff&   \d a + a_\wedge a+t(b)=0.
\end{align*}
Here we obtain the equation \eqref{eq:ham:curv1} as the vanishing of the fake curvature. Finally, writing now the condition \eqref{eq;ham:curv3} locally, we obtain:
\begin{align*}
(\partialH\omega_H)(\partial_i,\partial_j,\partial_k)=p^*\phi(\partial_i,\partial_j,\partial_k)\iff&\oint_{i,j,k} \Lie_{\partial_i+a_i} \omega_H^{j,k}=\phi_{i,j,k}\\
\iff&\oint_{i,j,k} \partial_i \omega_H^{j,k}+ \Lie_{a_i} \omega_H^{j,k}=\phi_{i,j,k}\\
\iff&  \d b +a_\wedge b=\phi.
\end{align*}
In the language of $2$-connections, this means that $\phi=\partialH \omega_H$ is the curvature $3$-form of the $2$-connection $(a,b)$.

We can now summarize the discussion of weakly hamiltonian fibration in terms of $2$-connections in the following way:

\begin{proposition}
Let $(M\to B,\omega, \Hor)$ be a weakly hamiltonian fibration. Then for any $x\in M$, there exists a neighborhood $\mathcal{U}$ and a local coordinate chart $p^{-1}(\mathcal{U})\simeq \mathcal{U}\times F$ for which we obtain the following correspondence:
 \begin{eqnarray*}
 \omega_V \in \Omega^2_{sp}(\Ver)  &\loc& \omega_F \in \Omega^2_{sp}(F),\\
\ \   h\in\Omega^1\bigl(B,\X_B(M)\bigr) & \loc& a\in\Omega^1\bigl(\U,\X(F)\bigr),\\
 \omega_H&\loc& b\in\Omega^2\bigl(\U,C^\infty(F)\bigr).
\end{eqnarray*}
In particular, the operator associated with the connection reads locally:
\begin{eqnarray*}  \hspace{35pt} \partialH\! & \loc &\d + a_\wedge.\end{eqnarray*}

Furthermore, in this chart, the equations \eqref{eq:closed2:2}-\eqref{eq:closed2:4} ruling a weakly hamiltonian fibration are those of a $2$-connection $(a,b)$ with values in $C^\infty(F)\to \X(F)_{\omega_F}$, whose fake curvature vanishes, and with curvature $3$-form $\phi$, namely:
\begin{eqnarray*}
\partialH\omega_V=0\hspace{22pt}&\iff& a\in \Omega^1(\mathcal{U},\X(F)_{\omega_F}),\\
\partialC\omega_V+\partialV\omega_H =0 \hspace{22pt}&\iff& \d a+a_\wedge a + t(b)=0,\\  
\partialH\omega_H=p^*\phi\hspace{10pt}&\iff & \d b+a_\wedge b=\phi.
\end{eqnarray*}
\end{proposition}

\begin{remark}
For instance, the Proposition \ref{prop:obstr:omegaclosed2} appears now as a well-known fact, namely that if the fake curvature vanishes, then the curvature $3$-form takes values in $\ker \d t$  (in our case $\ker \d t=\mathbb{R}$, as $F$ is assumed to be connected).
\end{remark}

\subsection{Classical Yang-Mills fields revisited}\label{sec:Classical:YangMills}

In this section, we shall generalize a universal construction for hamiltonian fibrations  due to Weinstein \cite{Wein3},  in order to obtain weakly hamiltonian fibrations. The Theorem \ref{thm:classical:yangmills} is a fundamental result, as it gives  a geometrical context to the so-called minimal coupling \cite{Stern}. See also \cite{GS} for other applications. 

\begin{theorem}\label{thm:classical:yangmills}
 Consider the following setting:
\begin{itemize}
 \item a principal $G$-bundle $P\to B$, with a principal connection $\theta$;
 \item an hamiltonian space $G \circlearrowright(F,\omega_F)\xrightarrow{J}\g^*$.
\end{itemize}
Then the associated bundle $M:=P{\scriptstyle \mathop{\times}\limits_{G}}F$ comes equipped with a closed
 $2$-form that fibres symplectically over $B$.
\end{theorem}

\begin{proof} The usual construction can be obtained by reduction, see \cite{Wein3}, \cite[Thm. 6.17]{MaSa}, \cite[ex. 2.2]{GLS}, \cite[Prop. 35.5]{GS} for more details. Here we will give a proof using connections, in order to emphasize the approach. 

First of all, notice that $\theta$ induces an Ehresmann connection on the associated bundle. On the other hand, the $G$-invariance of $\omega_F$
 ensures that there is a well-defined vertical symplectic form $\omega_V$ locally equivalent to $\omega_F$. It also responsible for $\omega_V$ to be invariant
  by the connection, so that $\partialH\omega_V=0$.

The term $\omega_H$ is obtained by pairing the moment map with the curvature:
\begin{equation}\label{eq:classical:yang:mills:cocurvature}\omega_H:=\langle J,\omega_\theta\rangle \in\Omega^2(B,C^\infty(M)).\end{equation}
The connection induced by $\theta$ is easily seen to have hamiltonian curvature, and we  recover the condition $\partialC\omega_V+\partialV\omega_H=0$. Furthermore, the equation $\partialH\omega_H=0$ follows from the equivariance of $J$.
\end{proof}

Let us now emphasize the following crucial points:
\begin{itemize}
 \item  it is the equivariance condition of the moment map $J$ which is responsible for the equation $\partialH\omega_H=0$ to hold,
 \item in the case where $J$ is not assumed to be equivariant, the argument used above in order to \emph{define} $\omega_H$ does not even hold.
\end{itemize}
In order to obtain a weakly hamiltonian fibration, this suggests to perform the  construction of Thm. \ref{thm:classical:yangmills} using a \emph{non-equivariant} moment map. Recall that for a non-equivariant map $J: F\to \g^*$, it is only required that:
$$-i_{\xi_F}\omega_F=\d \langle J,\xi \rangle,\quad\forall\in\g.$$ 

In that situation, $\Lambda^J(\xi,\eta):=\{\langle J,\xi\rangle,\langle J,\eta\rangle\}- \langle J,[\xi,\eta]\rangle$ defines a $2$-cocycle $\Lambda^J:\wedge^2 \g \to \mathbb{R}$. Here $\{\ ,\ \}$ denotes the Poisson bracket induced by $\omega_F$ on the space of functions,  and we assumed for simplicity that $F$ is connected; we refer for instance to  \cite{MaSa} and \cite{Ab-Mars} for these basic facts. More explicitly, $\Lambda^J$ as defined above has values in locally constant functions and satisfies:
$$\oint_{\eta,\xi,\mu}\Lambda^J([\eta,\xi],\mu)=0.$$
Stated in a different language, this means that $(J^*,\Lambda^J)$ defines a
$L_\infty$-morphism $\bigl(0\to \g\bigr)\longrightarrow \bigl(\mathbb{R}\to C^\infty(F)\bigr)$ between $2$ terms $L_\infty$-algebras:
$$\xymatrix@C=20pt@R=20pt{ 0 \ar[r]\ar[d]\ar[r]   & \mathbb{R}\ar[d]                &                                           & \mathbb{R}  \\
            \g \ar[r]^>>>>{J^*} & C^\infty(F)                              &              \wedge^2\g\ar[ur]^{\Lambda^J}    }      $$
where $\{J^*,J^*\} \xrightarrow{\Lambda^J} J^*[\ ,\ ]$ plays the role of the homotopy,  and $J^*:\g\to C^\infty(F)$ denotes the dual map to $J$. In view of section \ref{sec:local:symplectic}, it is quite remarkable that  non-equivariant moment maps force one to consider $\g$ and $C^\infty(F)$ as $2$-algebras, as opposed to  equivariant moment maps, which are usual Lie algebra morphisms $J^*:\g\to C^\infty(F)$. 

\begin{ex}\label{app:non-equivariant} Consider a central extension of Lie algebras $\mathfrak{z}\hookrightarrow \mathcal{E} \twoheadrightarrow \g$, and a coadjoint orbit $F\subset \mathcal{E}^*$. Then any splitting $\g\hookrightarrow \mathcal{E}$ produces a non-equivariant moment map $J:=\sigma^\star|_F$ obtained by restricting  $\sigma^*: \mathcal{E}^*\to \g^*$. The simplest relevant example is given by the non-trivial central extension of $se_2\ltimes \mathbb{R}^2$, with a coadjoint orbit $F$ that does not lie in $\g^*\subset \mathcal{E}^*$.
\end{ex}
Back to the problem of weakly hamiltonian fibrations, we can prove:
\begin{proposition}
 Consider the following setting:
\begin{itemize}
 \item a principal $G$-bundle $P\to B$, with a principal connection $\theta$;
 \item an symplectic action with non-equivariant map $G \circlearrowright(F,\omega_F)\xrightarrow{J}\g^*$.
\end{itemize}
Then the associated bundle $M:=P{\scriptstyle \mathop{\times}\limits_{G}}F$ admits a structure of weakly hamiltonian fibration.
\end{proposition}
\begin{proof}
First we build a vertical form $\omega_V$ locally equivalent to $\omega_F$, and invariant by the connection induced on $M$ as in the proof of \ref{thm:classical:yangmills}, so that the equation \eqref{eq:ham:curv1} is satisfied.

 Then there only remains to construct $\omega_H$ satisfying $\partialC\omega_V+\partialV\omega_H=0$. For this, we just use a partition of unity.
More precisely, we may consider an open cover $\{\U\}_{\U\in \C}$ of $B$, with local sections $\U \to P$ and $\Psi_\U:p^{-1}(\mathcal{U})=\mathcal{U}\times F$
 the corresponding local trivialization. Then we denote:
\begin{itemize} 
 \item $a=\sum_i a_i\d x_i\in \Omega^1(\mathcal{U},\g)$ the local connection form of $\theta$,
 \item $\omega_\theta^\U=\sum_{i,j}\bigl(\partial_i a_j-\partial_j a_i +[a_i,a_j]\bigr)\d x_i {}_\wedge \d x_j$ the local curvature $2$-form.
\end{itemize} 
We apply then \eqref{eq:classical:yang:mills:cocurvature} locally, namely we set $\omega_H^{\U}:=\langle \omega_\theta^\U, J \rangle\in \Omega^3(\U,C^\infty(F))$.
 Then  $\omega_H:=\substack{\sum \\ {\scriptscriptstyle \U\in \C}} \rho_\U.\Phi_\U^*( \omega_H^\U)$ satisfies: 
\begin{align*}
 \partialC\omega_V+\partialV\omega_H=& \partialC\omega_V+\partialV\ \sum_{\U\in \C}\rho_\U.\Phi_\U^*( \omega_H^\U),\\
                                     =& \partialC\omega_V+\sum_{\U\in \C}\rho_{\U}.\Phi_\U^*(\d_F \omega_H^\U),\\
                                      =&\partialC\omega_V+\sum_{\U\in \C}\rho_{\U}.\Phi_\U^* \bigl( -i_{\omega_\theta^\U}\omega_F\bigr),\\
                                       =&\partialC\omega_V-\sum_{\U\in \C}\rho_\U.i_{\Phi_\U^*\omega_\theta^\U}\Phi_\U^*(\omega_F)=0.
\end{align*}
Here, we used the fact that the functions $\rho_\U$ are basic, and that $\Phi_\U^*\omega_F=\omega_V$ and $\Phi_\U^*\omega_\theta^\U=C$ by construction.
\end{proof}

\begin{remark}
In \cite{To},  Toppan points out the appearance of "anomalies" in \emph{classical} dynamical systems. The reason for such phenomenons to occur comes from the fact that the symmetries of a classical system may produce conserved Noether charges whose Poisson bracket satisfies a centrally extended version of the original
 symmetry algebra. This is very appealing to us when comparing with the above description of weakly hamiltonian fibrations and the role of non equivariant moment maps.
 Yet it remains unclear to us how relate precisely both approaches.
\end{remark}

\section{Gauge symmetries of a closed 2-form}\label{sec:gaugegeneral2forms}
We saw in the previous section how to interpret the closure equations of a $2$-form in the case of symplectic fibrations. There, the vanishing of the term $\alpha_H$ played a crucial role. One may wonder if there exists similar interpretations in the more general
 situation where the fibres are no more assumed to be symplectic manifolds. In that situation, there is no justification for requiring that $\alpha_H$ to vanish as in Prop. \ref{prop:induced:symp:connection}.

In that case, the first difficulty comes from the equation \eqref{eq:closed2:2} where we have a connection that does not preserve the fibred $2$-forms $\omega_V$. The aim of this section is to explain that, although $\omega_V$ may not be preserved by the parallel transport, there is still a structure attached to $\omega_V$ which is preserved. This is related with Lie algebroids \cite{Pr} and their prequantizations, which we recall now.

 A \emph{Lie algebroid} $A\to M$
 is a vector bundle $A$  together with a map $\sharp:A\to TM$ called the anchor, and a bracket $[\hspace{5pt},\ ]$
 defined on the space $\Gamma(A)$ of sections of $A$ such that for any  $f\in C^\infty(M)$ and $\alpha, \beta \in\Gamma(A)$ of $A$,
 the following conditions hold:
\begin{align*} 
 [\alpha,f\beta]&=f[\alpha,\beta]+\Lie_{\sharp(\alpha)}(f),\\
 [\alpha,\beta]&=-[\beta,\alpha],\\
 { \oint} \bigl[[\alpha,\beta],\gamma\bigr]&=0.
\end{align*}
Recall also the notion of infinitesimal symmetry for a Lie algebroid:
\begin{definition}\label{def:derivation:Liealgebroid} A \emph{derivation}, or \emph{symmetry}, of a Lie algebroid $A\to M$ is an application
$D:\Gamma(A)\to\Gamma(A)$ together with a vector field $X\in\X(M)$ (called the \textbf{symbol} of the derivation) such that for 
any smooth function $f\in C^\infty(M)$ and sections $\alpha,\beta\in \Gamma(A)$, the following conditions are satisfied:
\begin{align*}
   D(f\alpha)            &= f D(\alpha)+(\Lie_X f)\alpha,\\
  D([\alpha,\beta])   &=  [D(\alpha),\beta]+[\alpha,D(\beta)]. \end{align*}\end{definition}
A derivation is the infinitesimal version of a Lie algebroid automorphism: any derivation
determines a vector field on  $A$ whose flow is a Lie algebroid automorphism \emph{et vice versa}
(see \cite[App.]{CrFe2}). We shall denote $\Der(A)$ the space of 
derivations, it carries a natural structure of Lie algebra
 whose bracket is given by the commutator $[D_1,D_2]=D_1\circ D_2-D_2\circ D_1$.
\subsection{Closed 2-forms and prequantization Lie algebroids} %
\label{sec:forms:algebroids}         %
A closed $2$-form can be described in terms of Lie algebroids as a $2$-cocycle with values in the natural representation of $TF$ on $F\times\mathbb{R}$ given by $X\triangleright  f=\Lie_X f$. Concretely, this means that given a closed $2$-form, we obtain a Lie algebroid structure on $A_{\omega_F}:=TF\times\mathbb{R}\to F$ by letting $\sharp(X,f)=X$ and:
\begin{equation*}
  \bigl[(X,f),(Y,g)\bigr]=(\left[X,Y\right],\Lie_X g-\Lie_Y f+\omega_F(X,Y)).
\end{equation*}

In this description, $A_{\omega_F}$ comes equipped with a splitting $TF\hookrightarrow A_{\omega_F}$ of the anchor map. Yet there is a slightly more conceptual approach,
 where one can get rid of this splitting, it is the one of a prequantization Lie algebroid \cite{Cr1,Rog1}:
 
 \begin{definition} A prequantization Lie algebroid is a Lie algebroid extension
\begin{equation}\label{ex:attiyahS1}
\mathbb{R}\times F\hookrightarrow A \twoheadrightarrow TF. 
\end{equation}
\end{definition}

Given a prequantization Lie algebroid, the choice of a splitting $\sigma:TF\to A$ induces a $2$-form $\omega_F\in\Omega^2(F)$ by letting
 $\omega_F(X,Y):=[\sigma(X),\sigma(Y)]-\sigma([X,Y])$. This form turns out to be closed because of the Jacobi identity on $A$. Furthermore, as easily checked, different splittings give rise to \emph{cohomologous} $2$-forms.

 This last point turns out to be crucial for our purposes, and a key point in this discussion. Indeed, a diffeomorphism that does not preserve a $2$-form still induces an isomorphism of the prequantization, provided the De Rham cohomology class of the form is preserved. The significance of \eqref{eq:closed2:2} should now appear clearly:
though the vertical form $\omega_V$ may not be preserved by a connection, the induced  \emph{prequantization structure} on the fibres is.

This fact does not help understanding \eqref{eq:closed2:3} and \eqref{eq:closed2:4} yet.
 To this end, we need to know more about the symmetries
 of prequantization Lie algebroids, which will be the purpose of the next section. 

\begin{remark} Let us briefly recall how prequantizations are related with $S^1$-principal bundles \cite{Cr1,Rog1}. When the group of periods of a $2$-form is a lattice, \emph{i.e.} when
 $\{\int_S \omega_F/[S]\in \pi_2(F)\}=a\mathbb{Z}$ for some $a\in\mathbb{R}$, the Lie algebroid $A_{\omega_F}$ integrates to a transitive Lie groupoid. If $F$ is simply connected, we obtain a $S^1$-principal bundle by considering the source-fiber 
 at a point, acted on by the isotropy $\mathbb{R}/a\mathbb{Z}=S^1$.

Then \eqref{ex:attiyahS1} is the corresponding Atiyah sequence and the choice of a splitting $\sigma$ is equivalent to a principal connection with curvature $\omega_F$. This motivated the terminology of \emph{prequantization} Lie algebroids in \cite{Cr1}. In fact, up to integrability issues, one may always think of a prequantization Lie algebroid as the infinitesimal counterpart of a $S^1$-principal bundle.
\end{remark}
\subsection{Symmetries of a prequantization}
In this section, we shall take a closer look at the space of symmetries of a prequantization Lie algebroid. Note that by 'symmetries', we mean here derivations as in the definition \ref{def:derivation:Liealgebroid}. Our description will closely follow the treatment  of \cite{BCG} for exact Courant algebroids. We shall prove the following:
\begin{theorem}\label{thm:der:2form}
The space of derivations of a prequantization Lie algebroid $F\times\mathbb{R}\hookrightarrow A \twoheadrightarrow TF$ fits into an exact sequence of Lie algebras:
 \begin{equation}\label{seq:der:2form:infalgebras}
   \Omega^1_{cl}(F)\hookrightarrow  \Der(A)\twoheadrightarrow\X(F).
 \end{equation}
More precisely, given a splitting of the anchor $\sigma:TF\to A$, we denote $\omega_F=[\sigma,\sigma]-\sigma[\ ,\ ]$ the corresponding closed $2$-form.  
There is an identification:
\begin{equation}\label{id:der:2form}\Der(A)\simeq\Bigl\{(X,\alpha)\in \X(F)\ltimes\Omega^1(F)|\Lie_X\omega_F=\d\alpha\Bigr\},
 \end{equation}
such that the Lie algebra structure on $\Der(A)$ and its infinitesimal action on $A \simeq TF \mathop{\ltimes}\limits_{\omega_F} \mathbb{R}$
respectively take the form:\vspace{-5pt}
\begin{align}
\bigl[(X,\alpha),(Y,\beta)\bigr]=&\bigl([X,Y],\Lie_X\beta-\Lie_Y\alpha \bigr),\\
\label{eqn:split:derivation:action:2form}(X,\alpha)\triangleright(Y,g)=&\bigl(\Lie_X Y ,\Lie_X g+\alpha(Y)\bigr).
\end{align}
\end{theorem}
\begin{proof}
We first choose $\sigma:TF\to A$ and let $\omega_F$ the corresponding $2$-form, so that we have an identification $A\simeq A_{\omega_F}$. Observe that a derivation $D$
 of $A_{\omega_F}=\Ver\times\mathbb{R}$ always covers an application $Y\mapsto \Lie_{X}Y$ on $\X(F)$, it follows that $D$ is necessarily of the
 form $D=(X,\alpha)$ as in \eqref{eqn:split:derivation:action:2form} for some $\alpha\in \Omega^1(F)$. Then it suffices to look at the conditions on
 $(X,\alpha)$ to be a derivation of the Lie bracket. We compute using \eqref{eqn:split:derivation:action:2form} that:
\begin{align*}
 (X,\alpha)\triangleright \bigl[(Y_1,g_1),(Y_2,g_2)\bigr]
 =\bigl([X,[Y_1,Y_2]],&\ \Lie_X\circ\Lie_{Y_1}g_2-\Lie_X\circ\Lie_{Y_2} g_1\\&\ \ +\Lie_X\omega_F(Y_1,Y_2)+\alpha([Y_1,Y_2])\bigr),\\
 \bigl[(X,\alpha)\triangleright (Y_1,g_1),(Y_2,g_2)\bigr]=\bigl([[X,Y_1],Y_2],&\ \Lie_{[X,Y_1]}g_2-\Lie_{Y_2}\circ\Lie_{X}g_1\\
                                                     &\ \ -\Lie_{Y_2}\alpha(Y_1)+\omega_F([X,Y_1],Y_2)\bigr),\\
\bigl[(Y_1,g_1),(X,\alpha)\triangleright (Y_2,g_2)\bigr]=\bigl([Y_1,[X,Y_2]],&\ \Lie_{Y_1}\circ\Lie_X g_2 -\Lie_{[X,Y_2]}g_1\\
                                                     &\ \ +\Lie_{Y_1}\alpha(Y_2)+\omega_F(Y_1,[X,Y_2]\bigr).
\end{align*}
By collecting the terms that do not cancel with each others, we obtain that $(X,\alpha)$ is a derivation \emph{iff} the following condition holds:\vspace{-6pt}
\begin{multline*}
 \Lie_X\, \omega_F(Y_1,Y_2)-\omega_F([X,Y_1],Y_2])-\omega_F(Y_1,[X,Y_2])\\=\Lie_{Y_1} \alpha(Y_2)-\Lie_{Y_2} \alpha(Y_1)-\alpha([Y_1,Y_2]),
\end{multline*}
which is precisely the expression for $(\Lie_X\omega_F)(Y_1,Y_2)=\d\alpha(Y_1,Y_2)$. The exactness of \eqref{seq:der:2form:infalgebras} then
 easily follows since $(0,\alpha)\in \Der(A)$ \emph{iff} $\d\alpha=0$.\end{proof}
\begin{theorem}\label{thm:adjoint:action:long:exact:sequence:2form}
For a prequantization Lie algebroid $F\times\mathbb{R}\hookrightarrow A \twoheadrightarrow TF$,
 the ``adjoint`` action $\ad:\Gamma(A)\xrightarrow{} \Der(A)$
induces the following exact sequence:\vspace{-2pt}
 \begin{equation*}\label{seq:der:2form}
   H^0(F)\hookrightarrow \Gamma(A) \xrightarrow{\ad} \Der(A)\twoheadrightarrow H^1(F).
 \end{equation*}
More precisely, using the identifications of Thm. \ref{thm:der:2form} the map $\ad$ takes the form:
$\ad_{(X,f)}=\bigl(X,i_X\omega_F-\d f\bigr).$
\end{theorem}
\begin{proof}
Once chosen a splitting $\sigma$, we simply have to check that 
${\ad_{(X,f)}(Y,g)=}$ ${(\Lie_X Y,\Lie_X g-\Lie_Y f+ \omega_F(X,Y))=(X,i_{X}\omega_F-\d f)\triangleright (Y,g).}$

We obtain a map $\Der(A)\to H^1(F)$ by letting $(X,\alpha)\mapsto [i_X\omega_F-\alpha]$ and exactness easily follows.
\end{proof}

\begin{remark}
Notice that the formula \eqref{eqn:split:derivation:action:2form} suggests to write a derivation $(X,\alpha)$ in the following matricial form:
$$(X,\alpha)=\left[\begin{matrix}
\Lie_X &  0 \\
\alpha & \Lie_X
\end{matrix}\right].$$
This makes it easily observed that $(X,\alpha)$ preserves the splitting \emph{iff} $\alpha=0$, in which case $\Lie_X\omega_F=0$
 (recall that $(X,\alpha)$ is submitted to the condition $\Lie_X\omega_F=\d\alpha$). Also, note that inner derivations take the matricial form
$$\ad_{(X,f)}\simeq
 \left[\begin{matrix}
\Lie_X &  0 \\
 i_X\omega_F-\d f& \Lie_X
\end{matrix}\right]$$
so the action of $(X,f)\in \Gamma(A)$ preserves the splitting \emph{iff} $i_X\omega_F=\d f$. In particular, if $X$ vanishes then $f$ shall be a (locally) constant function.
\end{remark}
\subsection{Fibred prequantizations}\label{sec:fibrewisedprequantization}
Given a fibration $M\to B$, the discussion of the previous section makes sense fibre-wise.
 One may consider fibred prequantization Lie algebroids,
 defined as  extensions of the vertical bundle:
$$ M\times\mathbb{R}\hookrightarrow A_V \twoheadrightarrow \Ver.$$
Then by choosing a splitting $\sigma_V:\Ver\hookrightarrow A_V$, we obtain an identification $A_V\simeq \Ver\times\mathbb{R}$, and a vertically closed $2$-form
 $\omega_V:=[\sigma_V,\sigma_V]-\sigma_V[\ ,\ ]$. Under this identification, the anchor becomes $\sharp(X_V,f)=X_V$  and the bracket reads:
\begin{equation*}
 \bigl[(X_V,f),(Y_V,g)\bigr]=\bigl(\left[X_V,Y_V\right],\Lie_{X_V} g-\Lie_{Y_V} f+\omega_V(X_V,Y_V)\bigr).
\end{equation*}
Furthermore, if $\sigma_V,\sigma_V':\Ver\hookrightarrow A_V$ are different splittings, then the corresponding $2$-forms $\omega_V$, $\omega'_V$ give
 \emph{cohomologous} vertical $2$-forms $\omega_V-\omega'_V=\partialV(\sigma_V-\sigma'_V)$.
 Recall now that $[\omega_V]_1$, seen as a section of $H^2(\Ver)\to B$, defines an element in the spectral sequence. In fact, we have by construction:
\begin{theorem}\label{thm:prequantization:spectral:element}
For a complete fibration $M\to B$, the space of isomorphism classes of fibred prequantizations identifies with $E_1^{0,2}=H^2(\Ver)$.
\end{theorem}
We now introduce a notion of symmetry for fibred prequantizations:
\begin{definition} A symmetry  $D$ of a prequantization $A_V$ with symbol $X$ is a map $D:\Gamma(A_V)\to \Gamma(A_V)$ satisfying the following conditions:
\begin{align*}
 D(f\alpha)&=f D(\alpha)+\Lie_X(f)\alpha,\\
 D[\alpha,\beta]&=[D\alpha,\beta]+[\alpha,D\beta].
\end{align*}\end{definition}
Considering the space $\Der(A_V)$ of symmetries of a prequantization $A_V$, the first observation is the following:
\begin{proposition}
The space of derivations of a fibred prequantization comes equipped with a map $\Der(A_V)\to \X(B).$ 
\end{proposition}
\begin{proof}\label{prop:der:fibredpreq:anchor}
It follows from the definitions of a symmetry that $\sharp_V(D\alpha)=[X,\sharp_V\alpha_V],$ for any $\alpha\in\Gamma(A_V)$. Since $\sharp_V$ is surjective,
 this implies that the symbol $X$ of any symmetry preserves the entire vertical bundle. It follows that $X$ is necessarily $p_*$-projectable.
\end{proof}
Thus, for  fibred prequantizations the flow of a symmetry commutes with $p_*$,  for this reason we shall refer
 to this space as \emph{gauge} symmetries and denote it by  $\Der_B(A_V)$
 rather than $\Der(A_V)$, though this is redundant by the preceding proposition.
\begin{definition}
 A fibred prequantization $A_V$ is said to have \emph{enough symmetries} if the map $\Der_B(A_V)\to \X(B)$ in Prop. \ref{prop:der:fibredpreq:anchor} is surjective. 
\end{definition}
In the case where $A_V$ does not admit enough symmetries, the prequantizations induced on different fibres may differ as Lie algebroids. In this paper, we will always assume that there are enough symmetries.
\begin{theorem}\label{thm:gauge:der:2form} Let $M\xrightarrow{p} B$ be a complete fibration. If a vertical prequantization
 $M\times\mathbb{R}\hookrightarrow A_V \twoheadrightarrow \Ver$ has enough symmetries, then the space of these symmetries fits into an exact sequence of Lie algebras:
 \begin{equation*}\label{seq:gauge:der:2form}
   \Omega^1_{cl}(\Ver)\hookrightarrow\Der_B(A_V)\twoheadrightarrow \X_B(M).
 \end{equation*}
More precisely, given a splitting $\sigma:\Ver\to A_V$ we denote $\omega_V=[\sigma,\sigma]-\sigma[\ ,\ ]$ the corresponding closed vertical  $2$-form.
 Then there is an identification:
\begin{equation}\label{gauge:sym:2form:split}\Der_B(A_V)\simeq\Bigl\{(X,\alpha_V)\in \X_B(M)\ltimes\Omega^1(\Ver)\ \bigl|\  \Lie_X\omega_V=\partialV\alpha_V\Bigr\},
 \end{equation}
such that the Lie algebra structure on $\Der_B(A_V)$ and its infinitesimal action on
$A_V \simeq \Ver {\scriptstyle\mathop{\ltimes}\limits_{\omega_F}} \mathbb{R}$ respectively take the following form:\vspace{-5pt}
\begin{align}
\label{eq:bracket:gauge:derivation:split}  \bigl[(X,\alpha_V),(Y,\beta_V)\bigr]=&\bigl([X,Y],\Lie_X\beta_V-\Lie_Y\alpha_V\bigr),\\ 
\label{eqn:gauge:derivation:2form}         (X,\alpha_V)\triangleright (Y_V,g)=&\bigl(\Lie_X Y_V, \Lie_X f+\alpha_V(Y_V)\bigr).
\end{align}
\end{theorem}
The proof can be omitted as it mimics point by point the one the Theorem \ref{thm:der:2form} by using a splitting  $\sigma_V:\Ver\to A_V$.
At that point, the following should be easy to prove as well:
\begin{corollary}
 A fibred prequantization $A_V$ admits enough symmetries \emph{iff} the equation \eqref{eq:closed2:2} admits a solution $\alpha_H$.
\end{corollary}
Note that the condition above is independent of the choice of the Ehresmann connection and splitting $\sigma_V$ chosen.
  Also, it is worth noticing that a prequantization $A_V \in E_1^{0,2}$ seen as an element in the spectral sequence
 (Thm. \ref{thm:prequantization:spectral:element}) has enough symmetries \emph{iff} it is a cocycle for $\d_1$.
\begin{theorem}\label{thm:gauge:adjoint:action:long:exact:sequence:2form}
 Let $M\xrightarrow{p} B$ be a fibration. For a fibred prequantization
  $M\times\mathbb{R}\hookrightarrow A_V \twoheadrightarrow \Ver$ with enough symmetries
 the ``adjoint`` action $\ad:\Gamma(A_V)\xrightarrow{} \Der_B(A_V)$
 induces an exact sequence of Lie algebras of the form:\vspace{-3pt}
\begin{equation*}\label{seq:der:2form:infinitver}
  H^0(\Ver)\hookrightarrow \Gamma(A_V) \xrightarrow{\ad} \Der_B(A_V)\twoheadrightarrow \X(B)\ltimes H^1(\Ver).
\end{equation*}
More precisely, under the identifications of Thm. \eqref{thm:gauge:der:2form} we have: 
\begin{equation}\label{eq:ad:vertical}\ad_{(X_V,f)}=\bigl(X_V,i_{X_V}\omega_V-\d f\bigr)\in\X(\Ver)\ltimes\Omega^1(\Ver). 
\end{equation}
\end{theorem}
Again, the proof is similar to that of the Theorem \ref{thm:adjoint:action:long:exact:sequence:2form} and can be omitted.

\subsection{Closed extensions as Lie algebroid fibrations}\label{closed:extensions1}
As explained in \cite{BZ}, a natural class of \emph{locally trivial} fibrations for Lie algebroids is given by abstract extensions of a tangent bundle $TB$:
$$A_V\hookrightarrow A_M \twoheadrightarrow TB,$$
under the completeness assumption. As the following result shows,  when  $A_V$ is a prequantization, we obtain a geometrical context to the Thm. \ref{thm:closed:extension} in the case of $2$-forms:
\begin{theorem}\label{thm:thurston:geometric:2forms}
Consider a complete fibration  $M\to B$ together with a closed $2$-form $\omega_V$ defined fibre-wise, and denote $A_{\omega_V}$ the associated fibred prequantization.
Then closed $2$-forms on $M$ whose cohomology class restricts to $[\omega_V]_1\in H^2(\Ver)$ are the same as Lie algebroid extensions $A_M$ of the form $A_V\hookrightarrow A_M \to TB$.
\end{theorem} 
\begin{proof}
Recall from \cite{Br} that, up to the choice of a complementary sub-bundle of $A_V$ in $A_M$, extensions of the form $A_V\hookrightarrow A_M \to TB$ are equivalent
 to couples $(\D,\omega)$ where $\D:\X(B)\to \Der_B(A_V)$ and $\omega\in\Omega^2(B, \Gamma(A_V))$ are submitted to the following compatibility conditions:
\begin{align}\label{eq:struc:extensions2form}
\Curv_\D=&\ad\circ\omega,\\
\D^\triangleright\omega=&0, 
\end{align}
where $\Curv_\D:=[\D,\D]-\D_{[\ ,\ ]}$ and $(\D^\triangleright\omega)(u,v,w):=\oint_{u,v,w} \D_u\omega(v,w)-\omega([u,v],w)$. In the case where $A_V$ is a vertical prequantization, one may take advantage of an identification $A_V\simeq \Ver\times \mathbb{R}$
 and use \eqref{gauge:sym:2form:split} in order to obtain a more explicit description for $(\D,\omega)$. We obtain that:
\begin{itemize}
 \item $\D$ is of the form $\D_u=\bigl(h(u),\alpha_H(u)\bigr)\in \X_B(M)\ltimes \Omega^1(\Ver)$ where $h(u)$ and $\alpha_H(u)$ satisfy
 $\Lie_{h(u)}\omega_V=\partialV\!\!\circ\alpha_H(u),$ we recover here \eqref{eq:closed2:2}; 
 \item $\omega$ is of the form $\omega(u,v)=\bigl(C(u,v),\omega_H(u,v)\bigr) \in\Gamma(\Ver)\ltimes C^\infty(M)$.
\end{itemize}
Here, $h(u)$ is the horizontal lift of a usual Ehresmann connection on $M\to B$, defined to be the symbol of the derivation $\D_u$. Then using the formula \eqref{eq:bracket:gauge:derivation:split} for the bracket of derivations, we can also compute $\Curv_\D$ more explicitly:
\begin{align*}\Curv_\D(u,v)&:=[\D_u,\D_v]-\D_{[u,v]}\\
                           &=\bigl([h(u),h(v)],\Lie_{h(u)}\alpha_H(v)-\Lie_{h(v)}\alpha_H(u)-\alpha_H(u,v)\bigr)\\&\quad\quad\quad\quad\quad-(h([u,v]),\alpha_H([u,v]))\\
                           &=\bigl([h(u),h(v)]-h([u,v]),\partialH\alpha_H(u,v)\bigr).
\end{align*}
It follows by applying the formula \eqref{eq:ad:vertical} for the adjoint action that the structure equation \eqref{eq:struc:extensions2form} reads:
$$\Curv_\D=\ad\circ\omega  \iff
 \left\{\begin{array}{l}\partialC\omega_V+\partialH\alpha_H+\partialV \omega_H=0,\\ C \text{ is the curvature of h}.
          \end{array}\right.
$$
Thus we recover here the equation \eqref{eq:closed2:3}. Finally, taking into account the way that $\Der_B(A_V)$ act on $A_V$ \eqref{eqn:gauge:derivation:2form} we easily obtain:
$$\D^\triangleright\omega=0\iff \partialC \alpha_H+\partialH\omega_H=0.$$  
This is precisely the equation \eqref{eq:closed2:4} and we conclude that $\omega_V\oplus\alpha_H\oplus\omega_H$ defines a closed $2$-form on $M$.

Reciprocally, given a closed $2$-form $\omega_M$ defined on the total space, we may consider the prequantization $A_{\omega_M}\hspace{-4pt}\to M$ induced on $M$, 
and compose the anchor with $p_*$ in order to obtain a Lie algebroid epimorphism $A_{\omega_M}\hspace{-4pt}\twoheadrightarrow TB$.
As easily seen, it has kernel $A_V=\Ver {\scriptstyle\mathop{\ltimes}\limits_{\omega_V}} \mathbb{R}$ so we have an extension:\vspace{-5pt}
$$A_{\omega_V}\hookrightarrow A_{\omega_M}\twoheadrightarrow TB.$$
By choosing an Ehresmann connection $h$, we can see the horizontal distribution as a complementary sub-bundle of $A_{\omega_V}=\Ver \times \mathbb{R}$
in $A_{\omega_M}=TM\times \mathbb{R}$. Then we apply the construction of \cite{Br}, namely: $\D$ and $\omega$ are obtained by letting
$\D_u:=[h(u),\ ]_{A_{\omega_M}}$ and $\omega(u,v):=[h(u),h(v)]_{A_{\omega_M}}\hspace{-4pt}-ħ([u,v])$. 
By doing this, we easily recover the expressions $\D=(h,\alpha_H)$ and  $\omega=(C,\omega_H)$. \end{proof}

\subsection{The 2-connection picture}\label{sec:2connections}
We now want to interpret the equations \eqref{eq:closed2:1}-\eqref{eq:closed2:3} in terms of  $2$-connections, as we did for weakly hamiltonian fibration in the Section \ref{sec:local:symplectic}. Again we shall recover \eqref{eq:closed2:3} as the vanishing of the fake curvature, and the left hand side of the equation \eqref{eq:closed2:4} as the curvature $3$-form.

To this purpose, we first use the Ehresmann connection in order to obtain a local diffeomorphism 
$\Psi_\U:p^{-1}(\mathcal{U})\to F\times \mathcal{U}$, where $\mathcal{U}$ is a open subset in $B$ diffeomorphic to $\mathbb{R}^n$ with  system of local coordinates $x=(x_1,\dots,x_n)$. We denote by  $\partial_i=\partial x_i$ the corresponding vector fields.

The connection and its curvature write locally as follows:
\begin{align*}
h(\partial_i)&\loc\partial_i+a_i,\text{ where }a_i:\mathcal{U} \to \X(F),\\
C(\partial_i,\partial_j)&=[h(\partial_i),h(\partial_j)]-h([\partial_i,\partial_j])\\ &\loc\partial_i a_j-\partial_j a_i+[a_i,a_j].
\end{align*}
By letting $a:=\sum_i a_i\d x_i\in \Omega^1(\mathcal{U},\X(F))$ we obtain the local connection form, and we recover above the formula $C\loc\d a+ a_\wedge a$.
Then we take local coefficients $\alpha_H\loc\sum_i\alpha_i\d x_i$, where $\alpha_i:\mathcal{U}\to \Omega^1(F).$ Using the description in the proof of Thm. \ref{thm:thurston:geometric:2forms}, we obtain the following local expression for $\D$:
$$\D_{\partial x_i}=(h(\partial_i),\alpha_H(\partial_i))\loc(\partial_i+a_i,\alpha_i).$$
Let us now set  $\alpha:=\sum_i \alpha_i\d x_i\in\Omega^1(\mathcal{U},\Omega^1(F))$. Here, the reader may probably expect $(a,\alpha)$
 to be 'the' local connection form of $\D$. There is however a basic issue that will prevent us to do this.
 Indeed, unlike in the symplectic case, $\omega_V$ depends on the basic variables,
 thus when writing the equation \eqref{eq:closed2:2} in local coordinates, we obtain:
\begin{equation}\label{wronglocal2form}
 \partial_i{\omega_V}+ \Lie_{a_i}\omega_V+\partialV\alpha_i=0.
\end{equation}
Though this equation establishes $(\partial_i+a,\alpha_i)$ as an element of $\Der(A_{\omega_V})$, the fact that $\partial_i\omega_V\neq 0$  does not allow to write that $(a,\alpha)\in \Omega^1(\U,\Der(A_{\omega_F}))$.

Instead, we can take advantage of \eqref{eq:integration:2form}, from which we deduce the existence of $\Delta_\U:\mathcal{U}\to\Omega^1(F)$ satisfying
$\omega_V+\partialV\Delta_\U=\omega_F.$
Using $\Delta_\U$, we can now rewrite the equation \eqref{wronglocal2form} in the following manner:
\begin{equation*}
 \Lie_{a_i}\omega_F=\partialV\bigl(\alpha_i+(\partial_i+\Lie_{a_i})\Delta_\U\bigr),
\end{equation*}
which precisely means that $(a,\alpha+\partialH\Delta_\U)$ takes values in $\Der(A_{\omega_F})$.
 This makes $A:=(a,\alpha+\partialH\Delta_\U)$ a much better candidate for what to be defined as the local connection form of $\D$.
 
We now look at the structure equations in local coordinates, for the sake of simplicity, we will assume first that the term $\Delta_\U$ vanishes.

\noindent\textbf{The case $\Delta_\U=0$:}

When writing down the local curvature form associated with $A$, we obtain:
\begin{align*}
   (\d A+A_\wedge A)(\partial_i,\partial_j )
&= \partial_i(a_j,\alpha_j)-\partial_j(a_i,\alpha_i)+[(a_i,\alpha_i),(a_j,\alpha_j)]_\g\\
&= \bigl(\partial_i a_j-\partial_j a_i,\partial_i \alpha_j-\partial_j \alpha_i\bigr)
                           +\bigl([a_i,a_j],\Lie_{a_i} \alpha_j-\Lie_{a_j} \alpha_i \bigr)\\
&= \bigl(\partial_i a_j-\partial_j a_i+[a_i,a_j],\partial_i \alpha_j
     +\Lie_{a_i} \alpha_j-   \partial_j \alpha_i     -\Lie_{a_j}\alpha_i\bigr)\\
&\loc \bigl(C(\partial_i,\partial_j),\partialH\alpha(\partial_i,\partial_j)\bigr).
\end{align*}
Here we recover the expression for $\Curv_\D$ in local coordinates, namely:
$$\d A + A_\wedge A\loc(C,\partialH \alpha_H))=\Curv_\D.$$
Let us now take local coefficients for $\omega_H\loc \sum_{i,j} \omega_H^{i,j}\d x_i{}_\wedge \d x_j$ and set: 
$$B:=(C_{i,j},\omega_H^{i,j})\d x_i{}_\wedge \d x_j\in \Omega^2(\mathcal{U},\X\ltimes C^\infty(F)).$$
Following the lines of Section \ref{closed:extensions1}, one may apply locally the formula \eqref{eq:ad:vertical} for the adjoint action, we easily obtain $\ad_B\loc (C, \partialV \omega_H+\partialC \omega_V )$, hence we observe that:
$$\eqref{eq:closed2:3}\  \partialC\omega_V+\partialH\alpha_H+\partialV\omega_H=0 \iff \d A+ A_\wedge A+\d t(B)=0.$$
In other words, we obtained locally a $2$-connection $(A,B)$ with values in the following crossed-module:
\begin{equation}\label{crossed:module:2form}\Bigr(\X{\scriptstyle\mathop{\ltimes}\limits_{\omega_F}}\!C^\infty(F)\xrightarrow{\d t} \X\ltimes\Omega^1(F)\Bigl)
=\Bigl(\Gamma(A_{\omega_F})\xrightarrow{\ad} \Der(A_{\omega_F})\Bigr)
\end{equation}
It has vanishing fake curvature provided equation \eqref{eq:closed2:3} holds. Let us now compute the corresponding curvature $3$-form, we obtain:
\begin{align*}   
(\d B+ A_\wedge B)_{i,j,k}
=&\oint_{i,j,k}\partial_k \bigl(C_{i,j},\omega_H^{i,j} \bigr)+(a_k,\alpha_k)\triangleright \bigl(C_{i,j},\omega_H^{i,j}\bigr)\\
=&\oint_{i,j,k} \bigl(\partial_k C_{i,j},\partial_k\omega_H^{i,j}\bigr)
                             +\bigl(\Lie_{a_k}{C_{i,j}},\Lie_{a_k}\omega_H^{i,j}+i_{C_{i,j}}\alpha_k\bigr)\\
=& \oint_{i,j,k} \bigl(0,\Lie_{\partial x_i+a_i}\omega_H^{j,k}+ i_{C_{i,j}}\alpha_k\bigr)           \\
\loc&\bigl(0, \partialH\omega_H+\partialC\alpha\bigr)(\partial_i,\partial_j,\partial_k).
\end{align*}
This way we recover the curvature $3$-form as the left hand side of \eqref{eq:closed2:4}  written in local coordinates.

\noindent\textbf{The case $\Delta_\U\neq 0$:}

In that situation, one shall think of the couple $(\Psi_{\mathcal{U}},\Delta_\U)$ as playing the role of a chart, which is also is consistent with the use of $2$-groups. Here, $\Delta_\U$ acts as a gauge transformation, sending $(\Psi_{\mathcal{U}})_*\omega_M$
 to $(\Psi_{\mathcal{U}})_*\omega_M+\d \Delta_\U$. In our context, this may be justified in the following manner: identifying $(\Psi_{\mathcal{U}})_*\omega_M$ with $\omega_M$ and $p^{}(\U)$ with $\U\times F$, the transformation $\omega_M\mapsto \omega_M+\d \Delta_\U$ reads:
\begin{align*}
 \omega_V  \mapsto&\ \omega_V+\partialV\Delta_\U=\omega_F\\
 \alpha_H    \mapsto&\hspace{4pt}\alpha_H+\partialH\Delta_\U \\
 \omega_H  \mapsto&\ \omega_H+\partialC \Delta_\U.
\end{align*}
Yet, the first two lines above are coherent with our motivation for defining $A=(a,\alpha+\partialH\Delta_\U)$ as the correct notion of connection $1$-form.
 
  This suggests that $B:=(C,\omega_H+\partialC \Delta_\U)$ should be the curvature $2$-form in the chart $(\Phi_\U,\Delta_\U)$. Indeed, a straightforward computation involving the relations \eqref{eq:commutation:relations} easily yields:
  $$\d A+[A_\wedge A]+\d t(B)=0\iff \eqref{eq:closed2:3} \text{ holds.}$$


 Furthermore, a similar computation shows that the curvature $3$-form $\d B+A_\wedge B$ vanishes precisely when the equation \eqref{eq:closed2:4} holds. Here of course, the action $\triangleright$ should be taken with respect to ${\omega_F}$ (rather than $\omega_V$) when pairing the tensor product $A_\wedge B$ in local coordinates.

We can  summarize the description of the closured equations $2$-form in terms of $2$-connections in the following way:
\begin{proposition}
Let $M\to B$ be a complete fibration and $\omega\in \Omega^2(M)$ be a closed $2$-form.
Then, given a complete connection $\Hor$ there is, for any $x\in B$, a neighborhood $\mathcal{U}\subset B$ and a local coordinate chart $(\Psi_{\mathcal{U}}, \Delta_\mathcal{U})$, where $\Psi_{\mathcal{U}}:p^{-1}(\mathcal{U})\simeq \mathcal{U}\times F$ and $\Delta_{\mathcal{U}}:\mathcal{U}\to \Omega^1(F)$ for which we obtain the following local correspondence:
\begin{eqnarray*}
\quad\quad\quad\quad\quad\quad\quad\D=(h,\alpha_H)                 &\loc&  A\in\Omega^1(\mathcal{U},\X\ltimes\Omega^1(F)),\\
\quad\quad\quad\quad\quad\quad\quad\omega=(C,\omega_H)             &\loc&   B\in\Omega^2(\mathcal{U},\X\ltimes C^\infty(F)).
\end{eqnarray*}
In particular, the covariant derivative associated with $\D$ reads locally: $$\quad\quad\quad\quad \D^\triangleright\ \quad \!\!       \loc \quad \d \ +A_\wedge $$ 

Furthermore, the closedness equations \eqref{eq:closed2:2}-\eqref{eq:closed2:4} for $\omega$ to be closed respectively read:
\begin{eqnarray*}
\quad\quad\partialH\omega_V+\partialV \alpha_H\hspace{42pt}=0\hspace{2pt}        &\iff & A\in \Omega^1(\mathcal{U},\Der(A_{\omega_F}))\\
\partialC\omega_V+\partialH \alpha_H +\partialV\omega_H =0 \hspace{2pt}&\iff & \d A+A_\wedge A + t(B)=0\quad\quad\quad\\  
\partialC \alpha_H + \partialH \omega_H=0                           &\iff & \d B+A_\wedge B=0.
\end{eqnarray*}
Namely, on the left side, we recover the structure equations for $(A,B)$ to be a $2$-connection with values in 
$$\,\Bigr(\X{\scriptstyle\mathop{\ltimes}\limits_{\omega_F}} C^\infty(F)\to  \X\ltimes\Omega^1(F)\Bigl)\ \,
=\ \Bigl(\Gamma(A_{\omega_F})\xrightarrow{\ad} \Der(A_{\omega_F})\Bigr)$$
and whose fake curvature $2$-form, and curvature $3$-form both vanish.
\end{proposition}

Of course, the bottom line is the following:

\begin{corollary}
Let $M\to B$ be a complete fibration and $\omega\in \Omega^2(M)$. Then the equation for $\d \omega=0$ is equivalent to the structure equations of a $2$-connection with values in a crossed module of the form $\Gamma(A_{\omega_F})\xrightarrow{\ad} \Der(A_{\omega_F})$, and whose fake curvature $2$-form and curvature $3$-form both vanish. Here, $A_F$ is a prequantization Lie algebroid.
\end{corollary}

\section{Closed 3-forms on a fibration}\label{sec:phi:decomposition}
We now apply the construction of the Section \ref{Sec:Leray:Serre} in order to study closed $3$-forms defined on the total space of a complete fibration $M\to B$, first focusing on $2$-plectic forms.

For a $3$-form $\Phi\in\Omega^3(M)$, the decomposition \eqref{eq:filtration:forms} reads:
\begin{equation}\label{eq:phi:decomposition}
 \Phi=\Phi_V\oplus{\alpha_H^\star}\oplus\omega^\star_{H}\oplus\Phi_{H}.
\end{equation}
More explicitly, in this expression:
\begin{itemize}
 \item $\Phi_V$ is a vertical $3$-form, \emph{i.e.} a $3$-form on the fibres: $\Phi_V\in \Omega^3(\Ver)$,
 \item ${\alpha_H^\star}$ is a form with values in vertical $2$-forms: $\alpha^\star_{H}\in\Omega^1(B,\Omega^2(\Ver))$,
 \item $\omega^\star_{H}$ is a $2$-form with values in vertical forms: $\omega^\star_{H}\in\Omega^2(B, \Omega^1(\Ver))$,
 \item $\Phi_{H}$ is a $3$-form with values in functions: $\Phi_{H}\in\Omega^3(B,C^{\infty}(M))$.
\end{itemize}
Moreover, $\Phi$ is closed \emph{iff} the following compatibility conditions hold:
\begin{eqnarray}
 \partialV\Phi_V                                         \hspace{120pt}             &=&0, \label{eq:phi:closed1}\\
 \partialH\Phi_V+\partialV{\alpha_H^\star}               \hspace{80pt}             &=&0,\label{eq:phi:closed2}\\
 \partialC\Phi_V+\partialH{\alpha_H^\star}+\partialV\omega^\star_H   \hspace{41pt} &=&0,\label{eq:phi:closed3}\\
                 \partialC{\alpha_H^\star}+\partialH\omega_H^\star+\partialV\Phi_H \hspace{1pt}    &=&0,\label{eq:phi:closed4}\\
                                           \partialC\omega_H^\star+\partialH\Phi_H &=&0.\label{eq:phi:closed5}
\end{eqnarray}
The first two equations above can be interpreted as follows:
\begin{itemize}
 \item $\Phi_V$ is vertically closed, we get a cohomology class in the vertical De Rham cohomology $[\Phi_V]_1\in H^3(\Ver)$. We will think of
  $[\Phi_V]_1$ as a section of a vector bundle over $B$;
 \item the covariant derivative of $\Phi_V$ takes values in exact forms for which ${\alpha_H^\star}$ provides primitives. It follows that $[\Phi_V]_1$ is
 a parallel section of $H^3(\Ver)$ as a flat vector bundle.
More precisely, given a complete vector field $u\in\X(B)$, we obtain by integrating \eqref{eq:phi:closed2} the following formula:
\begin{equation} \label{integration:holonomy:phi}
 \bigl(\phi^{h(u)}_t\bigr)_*\Phi_V=\Phi_V+\d_V\Bigl(\,\int_0^t (\phi^{h(u)}_s)_*\alpha^\star_H(u) \d s\Bigr). 
\end{equation}
\end{itemize}
We will postpone the interpretation of the other equations to a later section. Instead, let us first look at the case where $\Phi_V$ defines a $2$-plectic
 form on the fibres.

\subsection{2-plectic fibrations}
For higher dimensional fields theories, some models use multiplectic, rather than symplectic forms. For $n=2$, which corresponds to string theory, we obtain a categorification of symplectic geometry. This was applied for instance to the dynamics of  a classical string in \cite{BHR}, and it is natural to try to look at $2$-plectic fibrations, trying to reproduce the discussion of symplectic fibrations previously exposed. First, let us first briefly recall some basic notions concerning  $2$-plectic geometry.
\begin{definition}
A \textbf{$2$-plectic form} on a manifold $M$ is a closed $3$-form $\Phi$ which is non-degenerate, in the sense that the following property holds:
$$i_X\Phi=0 \Rightarrow X=0, \quad \forall X\in TM.$$
\end{definition}
Thus a closed $3$-form is $2$-plectic if the contraction induces an injection $TM\hookrightarrow \wedge^2 T^*M, X\mapsto i_X\Phi$. For a $2$-plectic form, the notion of hamiltonian function is replaced by hamiltonian $1$-forms:
\begin{definition}
 A $1$-form $\alpha$ is \textbf{hamiltonian} if there exists a vector field $X$ on $M$ such that:
$$-i_X\Phi=\d\alpha,$$
\end{definition}
Note that unlike in symplectic geometry, not every $1$-form yields a hamiltonian, as $\d \alpha$ shall lie in the image of the contraction. However, if $\alpha$ is, then the corresponding \emph{hamiltonian vector field} $X$ is unique. Moreover its flow preserves $\Phi$.

Let us now look at a basic notion of fibration for $2$-plectic forms: we want them to be fibration with 
the extra data of a $2$-plectic form on the fibres, such that those identify $2$-plectically with each other:
\begin{definition}
A  \textbf{fibration with $2$-plectic fibres} is a fibration $M\to B$ together with a vertically closed non-degenerate $2$-form $\Phi_V\in\Omega^3(\Ver)$:
\begin{align*}
 \partialV\Phi_V&=0,\\
\forall X_V\in \Ver,\ i_{X_V}\Phi_V&=0\  \Rightarrow X_V=0.
\end{align*}
\end{definition}
\begin{definition}
In that case, a \textbf{compatible}, or $2$-\textbf{plectic connection} is a connection whose parallel
 transport preserves $\Phi_V$:
\begin{equation*}
 \partialH \Phi_V=0.
\end{equation*}
\end{definition}
\begin{definition}
 A \textbf{$2$-plectic fibration} is a fibration with $2$-plectic fibres that admits a compatible connection.
\end{definition}
It is tempting here to reproduce the discussion of symplectic fibrations.
 For that matter, if we just start with a $3$-form $\Phi$ whose restriction to the fibres is $2$-plectic, then there is no direct analogue of Prop. \ref{prop:induced:symp:connection}. The problem is that there may not be a connection for which the component $\alpha_H^\star$ vanishes. Note however that if such a connection (\emph{i.e.} one with $\alpha_H^\star=0$) exists  then it is unique. We will be more restrictive and introduce the following notion, which will be enough for
   basic situations (see sub-section \ref{basicexamples}).

\begin{definition}\label{defn:fibres:2plectically}
 A  $3$-form $\Phi\in\Omega^3(M)$ defined on the total space of a fibration $M\to B$ is said to fibre $2$-plectically if the following conditions hold:
\begin{enumerate}[i)]
 \item $\text{Graph}(\Phi)\cap \Ver \oplus \wedge^2 \Ver^0=\{0\}$,
 \item $\text{Graph}(\Phi)\cap\wedge^2\Ver\oplus\Ver^0=\{0\}$.
\end{enumerate}\end{definition}
Here we considered successively $\Phi$ as an application $i_-\Phi:TM\to \wedge^2T^*M$ and $i_-\Phi:\wedge^2 TM\to T^*M$. With this definition, we can reproduce the discussion of symplectic fibration:
\begin{proposition}\label{prop:2plectic:connection}
 Let $\Phi\in\Omega^3(M)$ be a closed form that fibres $2$-plectically over $B$. Then:
\begin{itemize}
 \item $\Phi_V$ induces a $2$-plectic form on the fibres of $M\to B$,
 \item the $2$-plectic fibration $(M\to B,\Phi_V)$ admits a $2$-plectic connection,
 \item this connection has hamiltonian curvature.
\end{itemize}
\end{proposition}
\begin{proof}
 The fact that $\Phi_V$ is $2$-plectic on the fibres is a direct consequence of \eqref{eq:phi:closed1} and $i)$ in Def. \ref{defn:fibres:2plectically}. In order to see that there exists a $2$-plectic connection, it is enough by \eqref{eq:phi:closed2} to pick a connection such that
 $\alpha_H^\star$ vanishes. For this, we consider
$\eta:TM\to \wedge^2\Ver^*,\ X\mapsto i_X\Phi|_{\wedge^2\Ver},$
and observe that:
\begin{enumerate}[i)]
 \item $\text{Graph}(\Phi)\cap \Ver \oplus \wedge^2 \Ver^0=\{0\}\iff \ \ker \eta\cap \Ver =\{0\}$,
 \item $\text{Graph}(\Phi)\cap\wedge^2\Ver\oplus\Ver^0=\{0\}\iff  \im ^\top\eta\cap \Ver^0=\{0\}$.
\end{enumerate}
Taking duals, we see that the condition ii) is equivalent to $\ker \eta+\Ver=TM$ and we conclude that $TM=\ker \eta\oplus \Ver$. Thus by choosing $\Hor:=\ker \eta$, we obtain a connection for which the component $\alpha_H^\star$ of $\Phi$ vanishes, as easily seen by comparing ii) with \eqref{splitkforms}. This connection has hamiltonian curvature as can be seen by setting $\alpha_H^\star=0$ in \eqref{eq:phi:closed3}.
\end{proof}

\begin{remark}
Notice that the condition $\alpha_H^\star=0$ is stronger than just requiring the connection to preserve $\Phi_V$. 
In fact, we see by \eqref{eq:phi:closed2} that $\Phi_V$ is preserved if and only if $\alpha_H^\star$ takes values in closed forms. If we only required the restriction of $\Phi$ to be $2$-plectic (condition i)) one may discuss the existence of a compatible connection as follows:
\begin{itemize}
\item first we notice that given a fibration with $2$-plectic fibres, the existence of $\alpha_H^\star$ such that  \eqref{eq:phi:closed2} holds is independent
 of the choice of the connection. Indeed if $h'=h+\Delta$ is another connection (where $\Delta\in \Omega^1(B,\Gamma(\Ver)))$ then we have:
\begin{align*}(\partial_{H'}\Phi_V)(u)=&(\partialH \Phi_V)(u)+\Lie_{\Delta(u)}\Phi_V\\
                                      =&\partialV \alpha_H^\star(u)+\partialV\circ i_{\Delta(u)}\Phi_V
= \partialV (\alpha_H^\star+i_{\Delta}\Phi_V)(u).
\end{align*}
\item then we observe from the above computation that $h'$ is compatible \emph{iff} $\alpha_H^\star$ takes values
 in ''closed forms modulo sections of $\im \Phi_V^\flat$''.
\end{itemize}
In the sequel we will restrict ourselves to the case where a $2$-plectic connection exists.
 
\end{remark}

\subsection{Obstruction theory}
Let us now start with a $2$-plectic fibration and look for obstruction to a closed extension.
\begin{proposition}\label{prop:2plect:obstruction1}
 Let $(M\to B,\Phi_V)$ be a $2$-plectic fibration. Then there is class $[\partialC\Phi_V]_2\in H^2(B,H^2(\Ver))$ which is a first obstruction to the existence of a closed extension of $\Phi_V$.
\end{proposition}
\begin{proof}
 First we use the commutation relations to check that
$$-\partialV\circ\partialC\Phi_V=\partialC\circ\partialV \Phi_V+\partialH\circ\partialH\Phi_V=0,$$
Thus $[\partialC\Phi_V]_1$ is a well defined element in $\Omega^2(B,H^2(\Ver))$. Then we compute that
$\partialH[\partialC\Phi_V]_1=[\partialH\circ\partialC\Phi_V]_1=[\partialC\circ\partialH\Phi_V]_1=0,$ which ensures that the class $[\partialC\Phi_V]_2\in H^2(B,H^2(\Ver))$ is well defined.

This class vanishes \emph{iff} there exists $[\alpha_H^\star]_1\in\Omega^1(B,H^2(\Ver))$ such that:
 $$-[\partialC\Phi_V]_1=\partialH[-\alpha_H^\star]_1,$$
where $\alpha_H^\star\in\Omega^1(B,\Omega^2_{cl}(\Ver))$ takes values in \emph{closed} forms. This last equality lies in $\Omega^2(B,H^1(\Ver))$,
 therefore $-\partialC\Phi_V $ and $\partialH \alpha_H^\star$ need only to coincide up to a $2$-form with values in \emph{exact} $2$-forms.
 So there exists  $\omega_H^\star\in\Omega^2(B,\Omega^1(\Ver))$ such that:
$$-\partialC\Phi_V=\partialH\alpha_H^\star+\partialV\omega_H^\star.$$

Thus we see that the vanishing of the class $[\partialC\Phi_V]_2$ is equivalent to the existence
 of solutions $\alpha_H^\star,\omega_H$ to the equations \eqref{eq:phi:closed2} and \eqref{eq:phi:closed3}.\end{proof}

\begin{remark}
Notice that the  Proposition \ref{prop:2plect:obstruction1} may produce $3$-forms that do not fibre $2$-plectically.
From this point of view, the definition \ref{defn:fibres:2plectically} seems quite restrictive for the problem of existence of a closed extension.
\end{remark}

\begin{proposition}\label{prop:2plect:obstruction2}
Let $(M\to B,\Phi_V)$ be a $2$-plectic fibration. If the obstruction of Prop. \ref{prop:2plect:obstruction1} vanishes, then for any solution $\alpha_H^\star,\omega_H^\star$ as above, there is a second obstruction
 $[\partialH\omega_H^\star+\partialC\alpha_H^\star]_1\in\Omega^3(B,H^1(\Ver))$ to the existence of a closed extension of $\Phi_V$ with $\alpha_H^\star,\omega_H^\star$ as prescribed components.
\end{proposition}
\begin{proof}
So far, we have solutions to the equations \eqref{eq:phi:closed1}-\eqref{eq:phi:closed3}.
 In order to make sure that the class $[\partialH\omega_H^\star+\partialC\alpha_H^\star]_1$ is well defined,
 we check using the commutation relations that:
\begin{align*}\partialV\ (\partialH\omega_H^\star+\partialC\alpha_H^\star)
&=-\partialH\circ\partialV\omega_H^\star-\partialC\circ\partialV \alpha_H^\star-\partialH\circ\partialH\alpha_H^\star\\
&=-\partialH(\partialV\omega_H^\star+{\partialH}\alpha_H^\star)\\
&=-\partialH\circ\partialC\Phi_V\\
&=-\partialC\circ\partialH\Phi_V=0.
\end{align*}
Assume now that this class vanishes; by definition $[\partialH\omega_H^\star+\partialC\alpha_H^\star]_1=0$
 \emph{iff} there exists $\Phi_H\in\Omega^3(B,C^\infty(M))$ such that
$ \partialH\omega_H^\star+\partialC\alpha_H^\star=-\partialV\Phi_H,$
that is \emph{iff} \eqref{eq:phi:closed3} admits a solution.
\end{proof}

\begin{proposition}
Let $(M\to B,\Phi_V)$ be a $2$-plectic fibration with connected fibres. If the obstructions of Prop. \ref{prop:2plect:obstruction1} and \ref{prop:2plect:obstruction2} vanish, 
then for any solution $\alpha_H^\star,\omega_H^\star,\Phi_H$ as above,
 there is a third obstruction $[\partialH\Phi_H+\partialC\omega_H^\star]_1\in H^4(B)$. If it vanishes, then $\Phi_V$ admits a closed extension.\end{proposition}
\begin{proof}
We first check using the relations \eqref{eq:commutation:relations} that 
$\partialV(\partialH\Phi_H+\partialC\omega_H^\star)=0.$
Since $p$ has connected fibres, this implies that $\partialH\Phi_H+\partialC\omega_H^\star$ takes values in
 functions that are constant along the fibres, so it is is the pull-back of a $3$-form on $B$. Using \eqref{eq:commutation:relations} again, we obtain $\partialH (\partialH\Phi_H+\partialC\omega_H^\star)=0$.
 Since on pulled-back forms, $\partialH$ coincides with the De Rham differential on $B$, we obtain a class in $H^4(B)$.

This class vanishes \emph{iff}  $\partialH\Phi_H+\partialC\omega_H^\star=-\partialH \Phi_B$
 for some $\Phi_B\in\Omega^3(B)$. Moreover, we notice that \eqref{eq:phi:closed4} remains true if we replace
 $\Phi_H$ by $\Phi_H+\Phi_B$, while \eqref{eq:phi:closed5} then holds. We conclude that
 $\Phi=\Phi_V\oplus\alpha_H^\star\oplus\omega_H^\star\oplus\Phi_H$ defines a closed $3$-form on $M$.
\end{proof}

\subsection{Basic examples of 2-plectic fibrations}\label{basicexamples}

\subsubsection{Canonical 2-plectic fibrations}
 Here we want to reproduce the construction of canonical symplectic fibrations \cite[ex. 2.1]{GLS}. We shall prove that there is a $2$-plectic fibration associated with an arbitrary fibration $M\to B$, namely $\wedge^2 \Ver^*\to B$.

 First we need to recall the construction of canonical forms in $2$-plectic geometry: considered as a manifold, $\wedge^2 T^*\!F$ comes naturally endowed with a $2$-form $\omega_{F}^{2}$ defined by $(\omega_{F}^{2})_\eta:=(p^2_{F})^*\eta$, where 
        $\eta\in\wedge^2T^*\!F$ and $p^2_{F}:\wedge^2T^*\!F\to F$ denotes the projection.

 The canonical $3$-form $\Phi_F^2$ on $\wedge^2 T^*\!F$ is then obtained  by differentiation:
 $\Phi_F^{2}:=\d \omega_F^2$, and turns turns out to be $2$-plectic. Furthermore, it is easily checked that $\omega_F^2$ satisfies the following fundamental property:
 for any $2$-form $\eta\in \Omega^k(F)$ seen as an application $\eta:F\hookrightarrow \wedge^2 T^*F$, we have $\eta^*\omega_F^2=\eta.$ It follows that $\eta\in \Omega^k(F)$ is closed \emph{iff} $\eta^* \Phi^2_F=0$, which is a universal property.

We now look at a fibred version of this construction. Given $p:M\to B$, one may consider $\wedge^2 \Ver^*\!\!\to B$ as a fibration as well, obtained by composing $p$ with the projection  $p^2_{V}:\Lambda^2\Ver^*\to M$. Clearly, if $M\to B$ has fibre type $F$, then $\wedge^2 \Ver^*\!\!\to B$ has fibre type $\wedge^2 T^*\!F$. From this point of view $\wedge^2 \Ver^*\!\!\to B$ comes naturally equipped with a vertical $2$-form $\omega^2_V$ defined by:
$$(\omega_{V}^{2})_{\eta_V}:=(p^2_{F})^*\eta_V,
        \text{ where }\eta_V\in\wedge^2\Ver^*.$$
This form has the fundamental property that $\eta_V^*\omega_F^2=\eta_V$
for any $\eta_V\in\Omega^2(\Ver^*)$, seen as an application $\eta_V:M\to \wedge^2\Ver^*$. 
Thus it differentiates to a vertical $3$-form $\Phi^2_V:=\d^{V^2}\omega^{2}_V$ with the property that $\partialV\eta_V=0$ \emph{iff} $\eta_V^*\Phi_V^2=0$. Here, we denoted by $\d^{V^2}$ the vertical De Rham differential with respect to the projection $\wedge^2 \Ver^*\to B$.
 
  As the following theorem shows, $\Phi^2_V$ can always be extended into an \emph{exact} $3$-form on the total space $\wedge^2 \Ver^*$.
\begin{theorem}\label{thm:canonical:2plectic}
Consider a complete Ehresmann connection  $TM=\Hor\oplus \Ver$ on a fibration $M\xrightarrow{p} B$. 

If we denote $i_H:\wedge^2\Ver^*\hookrightarrow\Lambda^2T^*M$ the induced injection, then the pull-back $\Phi:=i_H^*\Phi^2_M$ of the fundamental $3$-form fibres $2$-plectically with respect to $\wedge^2\Ver^*\to B$.

 Moreover, if $\Hor^2$ denotes the connection induced on $\wedge^2\Ver^*\to B$ by the Proposition \ref{prop:2plectic:connection}, then the following hold:
\begin{itemize}
 \item  $\Phi$ decomposes into $\Phi:=\Phi^2_V\oplus\omega_{H^2}^\star$, where $\Phi_V^2$ denotes the canonical vertical $3$-form,
 \item the connection $\Hor^2$ naturally lifts $\Hor$, namely: a smooth path $\gamma:I\to B$ has holonomy $\wedge^2(\d\phi_\gamma^*)^{-1}:\wedge^2 T^*M_{\gamma_0}\to\wedge^2 T^*M_{\gamma_1}$ where  $\phi_\gamma:M_{\gamma_0}\to M_{\gamma_1}$, $M_{\gamma_i}:=p^{-1}(\{\gamma_i\})$ denotes the holonomy induced by $\Hor$,
 \item for any $u,v\in\X(B)$, $\omega^\star_H(u,v)$ is the canonical hamiltonian vertical $1$-form associated with the lift of $C(u,v)$ to $\wedge^2\Ver^*$.
\end{itemize}
\end{theorem}
\begin{proof}
Rather than pulling-back $\Phi^2_M$ directly, we shall first pull-back the fundamental $2$-form $\omega_M^2$, and then differentiate it.  In fact, since the injection $i_H:\wedge^2\Ver^*\hookrightarrow \wedge^2 T^*M$
 has image $\wedge^2 \Hor^0=\{\alpha\in\Omega^2(M),\ i_v\alpha=0\ \forall v\in\Hor\}$, we easily deduce the following simple expression for the pull-back of $\omega^2_M$:
 $$(i^*_H\omega^2_M)_{\eta_V}(a_\wedge b)=\eta_V\bigl(\d p_V^2(a)_V{}_\wedge\d p_V^2(b)_V\bigr).$$
 Here $a,b$ are vectors tangent to $\wedge^2\Ver^*$ at $\eta_V$, and $\d p_V^2(a)_V,\d p_V^2(b)_V$ denote the vertical components of $\d p_V^2(a),\d p_V^2(b)$ in the decomposition $TM=\Ver\oplus\Hor$. It follows easily that the pull-back $i^*_\H\omega^2_M$
 restricts to $\omega_V^2$ on the fibres of $\wedge^2\Ver^*\to B$. In fact, one can observe further that $i^*_\H\omega^2_M$ has a single component in 
 $\Omega^{0,2}$ provided we choose a connection on $\wedge^2\Ver^*\to B$ whose image by $\d p_V^2$ is $\Hor$, as is clearly the case for $\Hor^2$ described by the second item.
 
  We obtain this way a simple decomposition with respect to $\Hor^2$, namely
$  i^*_\H\omega^2_M=\omega_V^2\oplus 0 \oplus 0 \oplus 0.$ Thus, using the formulas for the differential we see that:
\begin{equation*}i^*_H \Phi_M^2=\d^{V^2}\omega_V^2 \oplus \d^{H^2}\omega^2_V \oplus \d^{C^2}\omega_V^2 \oplus 0.
\end{equation*}
 Observe yet that $\Phi$ has no component in $\Omega^3(B,C^\infty(\Lambda^2\Ver^*))$. Furthermore, we have $\d^{H^2}\omega^2_V=0$. In fact, $\omega_V^2$ is invariant by its universal properties since $\H^2$ lifts $\H$ (see \cite[Thm. 2.2.1]{GLS} in the case of symplectic fibrations). We end up with $\Phi=\Phi^2_V\oplus 0  \oplus \d^{C^2}\omega_V^2 \oplus 0$, which is enough to prove that $\Hor^2$ is indeed the connection as induced by Prop. \ref{prop:2plectic:connection}.
 
We now explain why the last statement is a consequence of the expression for $\omega_{H^2}^\star:=\d^{C^2}\omega_V^2$. Here we recall that $\omega_{H^2}^\star$ is supposed to prescribe hamiltonian $1$-forms for the curvature.

In symplectic geometry, any vector field on $F$ lifts into a hamiltonian vector field $X^1$
 on $T^*\!F$ with a canonical choice of a hamiltonian function $J_{X^1}:T^*\!F\to \mathbb{R}$  given by
 $J_{X^1}(\alpha)=\langle \alpha, X\rangle$. There is an alternative formula for $J$ which is $J_{X^1}:=i_{X^1}\Theta_F^1$ where $\Theta^1_F$ denotes the fundamental $1$-form on $T^*\!F$.

Similarly in $2$-plectic geometry, a vector field $X$ lifts into a $2$-plectic vector field $X^2$ on $\wedge^2T^*\!F$, with a canonical choice of hamiltonian $1$-form  $J_{X^2}\in \Omega^1(\wedge^2 T^*\!F)$ given by $\langle (J_{X^2})_\eta, u \rangle:=\eta(X,\d p (u))$.
An alternative notation for this equation is the following: $J_{X^2} =i_{X^2}\omega_F^2$.

Back to our situation, the vector field $C^2(u,v)$ on $\wedge^2\Ver^*$ lifts the vertical vector field $C(u,v)$ since $\Hor^2$ lifts $\Hor$. 
As such, it is a $2$-plectic vector field on each fibre, for which we have a canonical choice of a hamiltonian $1$-form given by
 $i_{C^2(u,v)}\omega_V^2=(\d^{C^2}\omega_V^2)(u,v)$. Therefore we see that $\omega_{H^2}^\star$ simply prescribes the canonical choice of
 hamiltonian $1$-form.
\end{proof}

\subsubsection{Compact semi-simple groups}
Consider a compact semi-simple Lie group $G$ with bi-invariant product $\langle\ ,\ \rangle_\g$ and denote $\Phi\in\Omega^3(G)$ the corresponding bi-invariant $2$-plectic form: $$\Phi_e(x,y,z)=\langle x,[y,z] \rangle_\g,$$
 where $x,y,z\in\g$ and $e\in G$ denotes the identity.
 
 If $K$ is a semi-simple subgroup of $G$, then $\Phi$ fibres $2$-plectically with respect to  $p:G\to G/K$. Indeed, $\Phi$ restricts to a $2$-plectic form on the left cosets as $K$ is assumed to be semisimple. Furthermore, since $G$ is compact, the product $\langle\ ,\  \rangle_\g$ is definite \emph{positive}. Thus by setting $\h:=\mathfrak{k}^\perp$ we obtain a splitting
$\g=\mathfrak{k}\oplus \h,$
where $\mathfrak{k}$ denotes the Lie algebra of $K$.  At the identity, we easily compute that:
$\Phi_e(h,k_1,k_2)=\langle h,[k_1,k_2] \rangle_\g=0.$
It immediately follows that $\Hor_g:=(L_g)_* \h$ defines an Ehresmann connection on $M:=G\to G/K$ for which the component $\alpha_H^\star$ vanishes.

For instance, $SU_3$ fibers $2$-plectically over $S^5$ with fiber type $SU_2$, as easily seen by seen by realizing $SU_2$ as the isotropy of $(1,0,0)\in \mathbb{C}^3$ for the standard representation.  More generally, $SU_n$ fibers over $S^{2n-1}$ with fiber-type $SU_{n-1}$.

\subsubsection{3-forms on bundles associated to hamiltonian actions}

The aim of this section is to produce closed $3$-forms on the total space of bundles associated to various notions of hamiltonian actions of a Lie group $G$.

In the sequel, $G$ will denote a Lie group and $P$ a principal $G$-bundle, for which we have fixed a principal connection $\theta:TP\to\g$.
We will denote $\omega_\theta$ its curvature. In each case below, $\theta$ descends to an Ehresmann connection on the associated bundle $M:=P\times_G F$,
as usual we denote $\Hor$ and $\partialH$ the corresponding horizontal sub-bundle and covariant derivative.

\begin{ex}\emph{Hamiltonian $2$-plectic fibrations.}
The following is a basic generalization of the Theorem \ref{thm:classical:yangmills} for $2$-plectic fibrations. We consider a situation where:
\begin{itemize}
 \item $G$ acts smoothly on a $2$-plectic manifold $(F,\Phi_F)$ with equivariant moment map  $J:\g\to\Omega^1(M)$, which means that $-i_{\xi_F}\Phi_F=\d J(\xi).$
\end{itemize}
Then the associated bundle $M$ comes equipped with a closed $3$-form that fibres $2$-plectically over $B$. In fact, all the arguments for $2$-forms easily carry through: first there is connection induced by $\theta$ on the associated bundle.
 Then, one can define a vertical $3$-form $\Phi_V$ by pulling-back $\Phi_F$ in a local coordinate chart. Since $\Phi_F$ is $G$-invariant, $\Phi_V$
  is well-defined and invariant by the induced connection.
   Furthemore  $\omega_H^\star(u,v):=\langle J ,\omega_\theta(u,v)\rangle$ is well-defined because $J$ is equivariant, and satisfies $\partialH\omega_H^\star=0$.

Here, we recover a decomposition $\Phi=\Phi_V\oplus 0 \oplus  \omega_H^\star \oplus 0$ completely analogous to the classical
 coupling of $2$-forms, where $\omega=\omega_V\oplus 0 \oplus \omega_H$.
\end{ex}

\begin{ex}\label{ex:quasi:twisted:presymplecic}\emph{Twisted presymplectic fibrations.}
Assume that  $G$ is equipped with a bi-invariant product $\langle\hspace{5pt} ,\ \rangle_\g$, and denote $\Phi^G\in \Omega^3(G)$ the corresponding invariant $3$-form. In that situation, there is a notion of $G$-valued moment map \cite[def 2.2]{AMM} for quasi-hamiltonian spaces (see also \cite[ex 3.14]{BuCr} where these spaces are interpreted in terms of Dirac realizations).

A quasi-hamiltonian space is a manifold $F$ equipped with a $2$-form $\omega_F$, together with an action of $G$ and a map $J:F\to G$ so that the following conditions are satisfied:
\begin{align*}
 & i)  &  \d \omega_F&=-J^*\Phi^G,  \\
 & ii) &   \Lie_{\xi_F}\omega_F&=0, \\
 & iii)& i_{\xi_F}\omega_F&=J^*\circ\sigma,\\
 & iv) & J:F\to &\ G\text{ is $G$-equivariant,}\quad\quad\quad\quad\quad\quad\\
 & v)  & \ker (\omega_F)_x&=\g_{J(x)}.                
\end{align*}
Here, $\sigma:\g\to \Omega^1(G),$ is given by $ \xi\to \frac{1}{2}\langle\xi_R+\xi_L,\ \rangle_\g$ where $\xi^L,\xi^R$ denote respectively left and right invariant vector fields associated wih $\xi$ on $G$.
In that situation, the associated bundle $M:=P\times_G F$ inherits of a vertical $2$-form $\omega_V$ induced by $\omega_F$ and we may just set $\omega:=\omega_V\oplus 0 \oplus 0$, then  consider its differential $\Phi:=\d \omega$ (similarly to the situation in the proof of Thm. \ref{thm:canonical:2plectic}).

Since $\omega_F$ is $G$-invariant by the condition $ii)$ we have $\partialH\omega_V=0$ therefore
 $\Phi$ decomposes into $\Phi=\Phi_V\oplus 0 \oplus \omega_H^\star \oplus 0$ with  $\Phi_V=\partialV\omega_V$ and $\omega_H^\star=\partialC\omega_V$. By the condition $iii)$ we easily see that $\omega_H^\star=\langle J,  \omega_\theta\rangle$. Notice here the analogy with the decomposition $\omega=\omega_V\oplus 0 \oplus \omega_H$ for hamiltonian fibrations.
\end{ex}

\section{Gauge symmetries of a closed 3-form}\label{sec:gaugegeneral3forms}
In this section, we will interpret the equations for a $3$-form to be closed in terms of gauge symmetries of a structure, namely an exact Courant structure defined fibre-wise, in the same spirit as in section \ref{sec:gaugegeneral2forms}. We shall recall the few basic facts we need concerning Courant algebroids, however the interested reader may refer to \cite{Roy3} concerning their role in mathematical physics, described in terms of graded supermanifolds.
\subsection{Closed 3-forms and Courant algebroids}
\label{sec:forms:courantalgebroids}
There is a geometric structure naturally associated with the De Rham class of a closed $3$-form, namely an exact Courant algebroid;
 these may be thought of as the infinitesimal analogue of a gerbe \cite{SeWe}. Let us briefly recall these notions.

An exact Courant algebroid $\mathcal{C}$ is a vector bundle fitting into an exact of vector bundles:
$$T^*F\hookrightarrow \mathcal{C}\stackrel{\sharp}{\twoheadrightarrow} TF,$$
together with a bracket $[\ ,\ ]:\Gamma(\C)\times \Gamma(\C)\to \Gamma(\C)$ defined on the space of smooth sections, and a fibre-wise
 symmetric non degenerate pairing $\langle\hspace{5pt},\ \rangle^+:\C\times\C\to \mathbb{R}$, such that the following conditions are satisfied:
\begin{align*}
[\alpha,[\beta,\gamma]]&=[[\alpha,\beta],\gamma]+[\beta,[\alpha,\gamma]],\\
\sharp[\alpha,\beta]   &=[\sharp,\alpha,\sharp\beta],\\
[\alpha,f\beta]        &=f[\alpha,\beta]+(\Lie_{\sharp \alpha}f) \beta,\\
\Lie_{\sharp\alpha}\langle\beta,\gamma\rangle^+ &=\langle[\alpha,\beta],\gamma\rangle^++\langle\beta,[\alpha,\gamma]\rangle^+,\\
[\alpha,\alpha]        &=\d^*\langle\alpha,\alpha\rangle^+.
\end{align*}
 for all $\alpha,\beta,\gamma\in\Gamma(\C)$. Here $d^*:=\frac{1}{2}\sharp^*\circ\d: C^\infty(F)\to \Gamma(\C)$ and we used $\langle\ ,\ \rangle$
 in order to identify $\C$ and $\C^*$.

 The basic motivation for introducing a geometric structure with such a set of axioms is that these encode  a $3$-class in the De Rham cohomology, called the \u{S}evera class, in a canonical way.
 The construction is similar to the one we exposed in section \ref{sec:forms:algebroids} for prequantization Lie algebroids: any splitting $TF\hookrightarrow \mathcal{C}$ gives
 an identification $\mathcal{C}=TF\oplus T^*F$ with anchor $\sharp(X,\alpha)=X$. Bracket and pairing respectively take the form:
\begin{align*}
  \bigl[(X,\alpha),(Y,\beta)\bigr]_{\Phi_F}=&\bigl([X,Y],\Lie_X \beta-i_Y\d\alpha+i_{X_\wedge Y} \Phi_F \bigr),\\
 \langle(X,\alpha),(Y,\beta)\rangle^+=&\frac{1}{2}\bigl(i_{X}\beta+i_{Y}\alpha\bigr),
\end{align*}
where $\Phi_F$ is a closed $3$-form (see \cite{BCG}) and different splittings give rise to \emph{cohomologous} $3$-forms. Reciprocally,
 given a closed $3$-form $\Phi_F$, the above formulas induce on $TF\oplus T^*\!F$ a structure of exact Courant algebroid which we may denote
 $\mathcal{C}_{\Phi_F}=TF {\scriptstyle\mathop{\ltimes}\limits_{\Phi_F}} T^*\!F.$

Again, the fact that different splittings lead to cohomological $3$-forms is crucial for us in view of \eqref{eq:phi:closed1}. Indeed, though the parallel transport might not preserve the vertical $3$-form $\Phi_V$, in order for \eqref{eq:phi:closed1} to hold, the exact Courant algebroid structure induced by $\Phi_V$ shall still be preserved.  The formula \eqref{integration:holonomy:phi} even suggests that the term $\alpha_H^\star$
shall be integrated as a part of a connection. In order to understand the remaining equations \eqref{eq:phi:closed2}-\eqref{eq:phi:closed5} we will proceed as in the case of $2$-forms, first looking at the symmetries of an exact Courant algebroid.

\subsection{Symmetries of an exact Courant algebroid}
In this section, we shall recall the results of Bursztyn-Cavalcanti-Gualtieri and refer to \cite{BCG} for more details. Recall that an
 \emph{infinitesimal symmetry} $D$ of a Courant algebroid $\C$ with symbols $X$ is an application $D:\Gamma(\C)\to \Gamma(\C)$ satisfying:
\begin{align*}
 D(f\alpha)&= fD(\alpha)+\Lie_X (f) \alpha\\
 D\bigl([\alpha,\beta]\bigr)&= [D\alpha,\beta\bigr]+\bigl[\alpha,D\beta\bigr],\\
 \Lie_X \langle\alpha,\beta\rangle^+&=\langle D\alpha,\beta\rangle^++\langle\alpha,D\beta\rangle^+.
\end{align*}
Notice that by the first equation above, $D$ has an associated linear vector field on $\C$ that projects onto $X$.
 The other equations express the compatibility with the structure on $\C$. The following is the analogue of Thm. \ref{thm:der:2form},
 it may be proved similarly:
\begin{theorem}\label{thm:der:3form}
The space of infinitesimal symmetries of an exact Courant algebroid $T^*F \hookrightarrow \mathcal{C} \twoheadrightarrow TF$
 fits into an exact sequence of Lie algebras:
\begin{equation*}\label{seq:der:3form}
   \Omega^2_{cl}(F)\hookrightarrow  \Der(\mathcal{C})\twoheadrightarrow \X(F).
\end{equation*}
More precisely, given a splitting of the anchor $\sigma:TF\to \mathcal{C}$ with corresponding closed $3$-form $\Phi_F$, there is an identification
\begin{equation*}\label{id:der:3form}\Der(\mathcal{C})\simeq\Bigl\{(X,\alpha^\star)\in \X(F)\ltimes\Omega^2(F)\bigl|\Lie_X\Phi_F=\d\alpha^\star\Bigr\}.
 \end{equation*}
For this identification, the Lie algebra structure on $\Der(\C)$ and its action on 
$\mathcal{C} \simeq TF {\scriptstyle\mathop{\ltimes}\limits_{\Phi_F}} T^*F$ respectively take the following form:\vspace{-5pt}
\begin{align}\bigl[(X,\alpha^\star),(Y,\beta^\star)\bigr]=&\,\bigl([X,Y],\Lie_X\beta^\star-\Lie_Y\alpha^\star\bigr),\\
\label{eq:split:derivation:action:3form}(X,\alpha^\star)\triangleright (Y,\beta)=&\,\bigl(\Lie_X Y ,\Lie_X \beta+i_Y\alpha^\star\bigr).
\end{align}
\end{theorem}
Then we have the analogue of Thm. \ref{thm:adjoint:action:long:exact:sequence:2form} for Courant algebroids:
\begin{theorem}\label{thm:adjoint:action:long:exact:sequence:3form}
For an exact Courant algebroid  $T^*F\hookrightarrow \mathcal{C}\twoheadrightarrow TF$,
 the ``adjoint`` action $\ad:\Gamma(\mathcal{C})\xrightarrow{} \Der(\mathcal{C})$
induces an exact sequence:\vspace{-4pt}
 \begin{equation*}\label{seq:ad:3form}
   \Omega^1_{cl}(F)\hookrightarrow \Gamma(\mathcal{C}) \xrightarrow{\ad} \Der(\mathcal{C})\twoheadrightarrow H^2(F).
 \end{equation*}
More precisely, using the identifications of Thm. \ref{thm:der:3form} we have:
$$\ad_{(X,\alpha)}=(X,i_X\omega_F-\d \alpha)\in\X(F)\ltimes\Omega^2(F).$$
\end{theorem}
Once chosen a splitting $\sigma$, one may use \eqref{eq:split:derivation:action:3form} in order to write a derivation $(X,\alpha^\star)$
 in the following matricial form: 
\begin{equation*}(X,\alpha^\star)\simeq \left[\begin{matrix}
\Lie_X &  0 \\
\alpha^\star & \Lie_X
\end{matrix}\right]\end{equation*}
Thus we see that $(X,\alpha^\star)$ preserves the splitting \emph{iff} $\alpha^\star=0$, 
 in which case $\Lie_X\Phi_F=0$. We also obtain the following expression for the inner derivations:
$$\ad_{(X,\alpha)}\simeq
 \left[\begin{matrix}
\Lie_X &  0 \\
 i_X\Phi_F-\d \alpha& \Lie_X
\end{matrix}\right]$$
Here, we can observe that the action of $(X,\alpha)\in \Gamma(\mathcal{C})$ preserves the splitting \emph{iff} $i_X\Phi_F=\d \alpha$.
In particular if $X$ vanishes, then $\alpha$ shall be closed. 

\subsection{Fibred Courant algebroids}
Given a fibration $M\to B$, the discussion of the previous section can be generalized fibre-wise.
 One shall consider vertical exact Courant algebroids $\C_V$, defined as vector bundle extensions of the vertical bundle:\vspace{-5pt}
$$ \Ver^*\hookrightarrow \mathcal{C}_V \stackrel{\sharp_V}{\twoheadrightarrow} \Ver.$$
together with a bracket $[\hspace{5pt},\ ]:\Gamma(\CV)\times \Gamma(\CV)\to \Gamma(\CV)$ defined on the space of smooth sections  and a fibre-wise symmetric
 non degenerate pairing $\langle\ ,\ \rangle^+:\CV\times\CV\to \mathbb{R}$, such that the following conditions are satisfied:
\begin{align*}
\bigl[\alpha_V,[\beta_V,\gamma_V]\bigr]&=\bigl[[\alpha_V,\beta_V],\gamma_V\bigr]+\bigl[\beta_V,[\alpha_V,\gamma_V]\bigr],\\
\sharp_V[\alpha_V,\beta_V]   &=[\sharp_V\alpha_V,\sharp_V\beta_V],\\
[\alpha_V,f\beta_V]        &=f[\alpha_V,\beta_V]+\Lie_{\sharp_V \alpha_V}(f) \beta_V,\\
\Lie_{\sharp_V\alpha_V}\langle\beta_V,\gamma_V\rangle^+ &=\bigl\langle[\alpha_V,\beta_V],\gamma\bigr\rangle^++\bigl\langle\beta_V,[\alpha_V,\gamma_V]\bigr\rangle^+,\\
[\alpha_V,\alpha_V]        &=\d^*\langle\alpha_V,\alpha_V\rangle^+,
\end{align*}
 for all $\alpha_V,\beta_V,\gamma_V\in\Gamma(\C_V)$. Here $d^*:=\frac{1}{2}\sharp_V^*\circ\d: C^\infty(M)\to \Gamma(\CV)$ and we used $\langle\ ,\ \rangle^+$ in order to identify $\CV$ and $\CV^*$.

Then it is a routine computation to check that vertical exact Courant algebroids canonically encode a $3$-cohomology class in the vertical De Rham cohomology
 in the expected way, namely: any splitting of the anchor $\Ver\hookrightarrow\CV$ gives an identification $\CV=\Ver\oplus \Ver^*$
 with bracket $\sharp_V(X_V,\alpha_V)=  X_V.$ Anchor and pairing are given by the following formulas:
\begin{align*}
 \bigl[(X_V,\alpha_V),(Y_V,\beta_V)\bigr]_{\Phi_V}&=\bigl([X_V,Y_V],\Lie_{X_V} \beta_V-i_{Y_V}\partialV \alpha_V+i_{X_V{}_\wedge Y_V} \Phi_V \bigr)\\
\bigl\langle(X_V,\alpha_V),(Y_V,\beta_V)\bigr\rangle^+&=\frac{1}{2}\bigl(i_{X_V}\beta_V+i_{Y_V}\alpha_V\bigr),
\end{align*}
where $\Phi_V\in\Omega^3(\Ver)$ is a vertically closed $3$ form, and different splittings give rise to \emph{cohomologous}
 vertical $3$-forms. Reciprocally, given a vertically closed $3$-form $\Phi_V$, the above formulas induce on $\Ver\oplus \Ver^*$
 a structure of vertical exact Courant algebroid which we will denote  $\mathcal{C}_{\Phi_V}=\Ver {\scriptstyle\mathop{\ltimes}\limits_{\Phi_V}} \Ver^*.$

Recall now that $[\Phi_V]_1$, as a section of $H^3(\Ver)\to B$, defines an element in the Leray-Serre spectral sequence. In fact, we just proved:
\begin{theorem}\label{thm:fibredcourant:spectralelement}
For a fibration $M\to B$, the space of isomorphism classes of fibred exact Courant algebroid structures identifies with $E_1^{0,3}=H^3(\Ver)$.
\end{theorem}
\subsection{Symmetries of fibred Courant algebroids}\label{sec:fibrewisedcourant}
We now take a closer look at the space of symmetries of a fibred exact Courant algebroid. The discussion will be similar as in
 Section \ref{sec:fibrewisedprequantization} for fibred prequantizations.

\begin{definition}\label{def:der:CV} A \emph{derivation} $D$ of $\CV$ with symbol $X\in\X(M)$ is an application $D:\Gamma(\CV)\mapsto \Gamma(\CV)$ satisfying:
\begin{align*}
 D(f\alpha)&= f D \alpha+\Lie_X (f) \alpha\\
 D\bigl([\alpha,\beta]\bigr)&= [D\alpha,\beta\bigr]+\bigl[\alpha,D\beta\bigr],\\
 \Lie_X \langle\alpha,\beta\rangle^+&=\langle D\alpha,\beta\rangle^++\langle\alpha,D\beta\rangle^+.
\end{align*}\end{definition}
Because of the first condition above, we have $\sharp_V(D(\alpha))=[X,\sharp_V\alpha]$, thus, as in the Prop. \ref{prop:der:fibredpreq:anchor}, the surjectivity of the anchor map  implies that:
\begin{proposition}\label{prop:der:fibredcourant:project}
 The space of derivations $\Der(\C_V)$ of a fibred exact Courant algebroid comes canonically equipped with a map $\Der(\CV)\to \X(B)$. 
\end{proposition}
For this reason, we will always refer to the space of symmetries of a fibred prequantization as \emph{gauge} symmetries, and denote this space  $\Der_B(\C_V)$
 rather than $\Der(\C_V)$ though this is redundant by the preceding proposition.
\begin{definition}
 A fibred exact Courant algebroid $A_V$ has \emph{enough symmetries} if the map $\Der_B(A_V)\to \X(B)$ in Prop. \ref{prop:der:fibredcourant:project} is surjective. 
\end{definition}

In the case where $\C_V$ does not admit enough symmetries, the Courant algebroid structures induced on different fibres may
 not even be isomorphic so we will assume that there are enough symmetries. Then we can repeat the reasoning of Thm. \ref{thm:der:3form}.
 We obtain the following description for symmetries:

\begin{theorem}\label{thm:gauge:der:3form} Let $M\xrightarrow{p} B$ be a fibration.
 If a fibred exact Courant algebroid $\Ver^*\hookrightarrow \mathcal{C}_V \twoheadrightarrow \Ver$ has enough symmetries, then the space of these symmetries fits
 into an exact sequence of Lie algebras:
 \begin{equation*}\label{seq:gauge:der:3form}
   \Omega^2_{cl}(\Ver)\hookrightarrow\Der_B(\C_V)\twoheadrightarrow \X_B(M).
 \end{equation*}
More precisely, given a splitting $\sigma:\Ver\to \mathcal{C}_V$ with corresponding closed vertical
  $3$-form  $\Phi_V$, there is an identification
\begin{equation*}\label{gauge:sym:3form:split}\Der_B(\mathcal{C}_V)\simeq\Bigl\{(X,\alpha^\star_V)\in \X_B(M)\ltimes\Omega^2(\Ver), \Lie_X\Phi_V=\partialV\alpha_V^\star\Bigr\},
 \end{equation*}
such that the Lie algebra structure on $\Der_B(\CV)$ and its action on $\mathcal{C}_V \simeq \Ver {\scriptstyle\mathop{\ltimes}\limits_{\Phi_F}} \Ver^*$ respectively take the following form:
\begin{align}
\bigl[(X,\alpha_V^\star),(Y,\beta_V^\star)\bigr]&=\bigl([X,Y],\Lie_X\beta^\star_V-\Lie_Y\alpha^\star_V\bigr),
\label{eq:bracket:gauge:der:split:3form}\\
(X,\alpha_V^\star)\triangleright (Y_V,\beta_V)&=(\Lie_X Y_V, \Lie_X \beta_V+i_{Y_V}\alpha_V^\star).
\label{eq:gaugedersplit3form}
\end{align}
\end{theorem}
It immediately follows:
\begin{corollary}
 A fibred exact Courant algebroid $\CV$ admits enough symmetries \emph{iff} the equation \eqref{eq:phi:closed2} admits a solution $\alpha_H^\star$.
\end{corollary}
Recall here that the existence of solutions to \eqref{eq:phi:closed2} is independent of the Ehresmann connection and splitting $\sigma_V$ chosen.
 Furthermore, using the characterization in Thm. \ref{thm:fibredcourant:spectralelement} it is easily seen that a fibred exact Courant
 algebroid $\CV \in E_1^{(0,3)}$ has enough symmetries \emph{iff} it is a cocycle for $\d_1$.
\begin{theorem}\label{thm:gauge:adjoint:action:long:exact:sequence:3form}
Let $M\xrightarrow{p} B$ be a fibration. For a fibred exact Courant algebroid $\Ver^*\hookrightarrow \mathcal{C}_V \twoheadrightarrow \Ver$ with enough symmetries,
 the ``adjoint`` action $\Gamma(\mathcal{C}_V)\xrightarrow{\ad} \Der_B(\mathcal{C}_V)$
induces an exact sequence of the form:
 \begin{equation*}\label{seq:ad:ver:3form}
  \Omega_{cl}^1(\Ver)\hookrightarrow \Gamma(\CV) \xrightarrow{\ad} \Der_B(\mathcal{C}_V)\twoheadrightarrow \X(B)\ltimes H^2(\Ver).
   \end{equation*}
More precisely, using the identification of Thm. \ref{thm:gauge:der:3form}, we have:
\begin{equation}\label{eq:ad:vertical:3form}\ad_{(X_V,\alpha_V)}=(X_V,i_{X_V}\Phi_V-\partialV\alpha_V)\in\X(\Ver)\ltimes\Omega^2(\Ver). 
\end{equation}
\end{theorem}

\subsection{The 3-connection picture}
In this section, we shall describe the equations \eqref{eq:phi:closed1}-\eqref{eq:phi:closed4}
 as those of a $3$-connection \cite{MartinsPicken3}. Such connections take values in a $2$-crossed module,
 that is, a complex of Lie algebras:\vspace{-2pt}
\begin{equation*}\mathfrak{l}\xrightarrow{\delta}\mathfrak{e}\xrightarrow{\partial}\mathfrak{g},
\end{equation*}
with left actions of $\g$ on $\mathfrak{l}$ and $\mathfrak{e}$ by Lie algebra morphisms and a pairing  $\mathfrak{e}\times  \mathfrak{e} \to \mathfrak{l}$,
 called the Pfeiffer lifting. We refer to \cite{MartinsPicken3} for the compatibility relations. The basic example of a $2$-crossed
 module is given by the automorphisms of a crossed module \cite[Sec. 1.2.7]{MartinsPicken3}. There is a notion of connection with values in a $2$-crossed module, defined as a triple $(a,m,\theta)$ where:
 $a\in\Omega^1(B,\g),$
 $m\in \Omega^2(B,\mathfrak{e}),$
 $\theta\in\Omega^3(B,\mathfrak{l}).$
Furthermore, in \cite{MartinsPicken3} Martins-Picken show that under the following conditions:
\begin{align}
 \partial m&=da +a_\wedge a,\label{eq:3connection:1} \\
\theta&=d m+ a_\wedge m,\label{eq:3connection:2} \\
\delta \theta&=d m+ a_\wedge m,\label{eq:3connection:3}
\end{align} there is a well-defined notion of $3$-dimensional holonomy, defined on the Gray $3$-groupoid of the base. They described the Wilson $3$-sphere of observables in this manner. In that context, the $3$-curvature $4$-form is given by:
\begin{equation*}\Theta=d \theta+ a_\wedge \theta +m_\wedge m.
\end{equation*}
In our situation, we can recover the above structure equations, though the local connection forms do not take values in a $2$-crossed module. Instead, we shall consider the following complex:
\begin{equation}\label{symDGLA}C^\infty(F)\to \Gamma(\C) \to \Der(\C),
\end{equation}
Recall how the structure is given once a splitting $TF\hookrightarrow \mathcal{C}$ was chosen. The above exact sequence takes the explicit form:
$$\xymatrix@R=2pt{
C^\infty(F) \ar[r]^<<<<{\partial} &
              \X(F){\scriptscriptstyle\mathop{\ltimes}\limits_{\Phi_F}}\Omega^1(F)\ar[r]^>>>>>{\ad}
                                                                 &\X(F)\ltimes\Omega^2(F)_{\Phi_F}   \\
            f\quad \ar@{|->}[r]    &\quad (0,\d f)\quad\\
                              &\quad (X,\alpha)\quad\ar@{|->}[r]       &   \quad  (X,i_X\Phi_F-\d\alpha)}$$
Moreover, the action of $\Der(\mathcal{C})$ is given by:
\begin{itemize}
 \item $(X,\omega)\triangleright  f=\Lie_X f$,
 \item $(X,\omega)\triangleright  (Y,\alpha)=([X,Y], \Lie_X \alpha+i_Y\omega)$,
\end{itemize}
and we have a pairing $ \mathcal{C}\times  \mathcal{C} \to C^\infty(F)$ given by
$$\langle (X,\alpha),(Y,\beta)) \rangle^{+} =\frac{1}{2}\bigl( i_X\beta+ i_Y\alpha\bigr).$$

In order to recover the equations \eqref{eq:3connection:1}-\eqref{eq:3connection:3}, it is simpler to first work globally, as we did for $2$ forms in the section  \ref{closed:extensions1}. Alternatively, these computations may be done in local coordinates as in the section \ref{sec:2connections}, provided one takes the appropriate notion of chart, \emph{i.e.} integrating
\eqref{integration:holonomy:phi} in order to obtain some $\Delta^\star\in\Omega^1(\mathcal{U},\Omega^2(\Ver))$ acting as a gauge transformation. By doing this, we obtain the following correspondence in a straightforward way:

\begin{proposition}
Let $M\to B$ be a complete fibration and $\Phi\in \Omega^3(M)$ be a closed $2$-form.
Then, given a complete connection $\Hor$ there is, for any $x\in B$, a neighborhood $\mathcal{U}\subset B$ and a local coordinate chart $(\Psi_{\mathcal{U}}, \Delta_\mathcal{U}^\star)$, where $\Psi_{\mathcal{U}}:p^{-1}(\mathcal{U})\simeq \mathcal{U}\times F$ and $\Delta^\star_{\mathcal{U}}:\mathcal{U}\to \Omega^2(F)$, for which we obtain the following local correspondence:
\begin{eqnarray*}
\hspace{50pt}\D=h\oplus {\alpha_H^\star} &\loc & a\in\Omega^1(B,\g) \\
(C,\omega_H^\star)& \loc &  m\in \Omega^2(B,\mathfrak{e})           \\
\Phi_H &\loc &\theta\in\Omega^3(B,\mathfrak{l}) 
\end{eqnarray*}

In particular, on both sides covariant derivative and curvature take the form:
\begin{eqnarray*}
 \D^\triangleright&\loc& d\ +a_\wedge  \\
 \text{\emph{Curv}}_\D=(C,\partialH{\alpha_H^\star})&\loc& d a+ a_\wedge a.
 \end{eqnarray*}

Under this correspondence, the structure equations \eqref{eq:3connection:1} \eqref{eq:3connection:2}  for $\Phi$ to be closed are the structure equations of a $3$-connection with values in $C^\infty(F)\to \Gamma(\C) \to \Der(\C),$ namely:
\begin{eqnarray*}
\partialC\Phi_V+\partialV\omega_H^\star=-\partialH{\alpha_H^\star}
&\iff & \ad_{(C,\omega_H^\star)}=\text{\emph{Curv}}_\D \\
                                         &\loc & \partial m=da +a_\wedge a,  \\
 -\partialV\Phi_H=\partialC{\alpha_H^\star}+\partialH\omega_H^\star
& \iff & -\partialV\Phi_H=\D^\triangleright(C,\omega_H^\star)   \\
                                         &\loc &\delta \theta=d m+ a_\wedge m  .
\end{eqnarray*} 
Finally, we obtain the curvature $4$-form as being precisely the left hand side of the equation \eqref{eq:phi:closed5}
\begin{eqnarray*} 
 \Theta=  \partial_H\Phi_H+\partialC \omega_H^\star&\loc&\Theta=d \theta+ a_\wedge \theta +m_\wedge m 
\end{eqnarray*}              
\end{proposition}

\begin{remark} Here, though we obtain the structure equations of a $3$-connection, notice that \eqref{symDGLA} does not carry the structure of a $2$-crossed module. This is not so surprising. Indeed, crossed-modules correspond to a \emph{strictly} categorified version of a Lie algebra \cite{BaezCrans}, for which the space of symmetries is a $2$-crossed-module \cite{Cond}, \cite[Sec. 1.2.7]{MartinsPicken3}. In our situation, \eqref{symDGLA} plays the role of symmetries \cite{Uribe} for $C^\infty(F)\to \Gamma(C)$, the which is not a crossed module, but rather a (semi-strict) Lie $2$-algebra \cite{BaezCrans} \cite{Royt2} \cite{Rog1}.
\end{remark}

\section*{Acknowledgements} 
 Amongst others, the author would like to thank warmly Camille Laurent-Gengoux, Marco Zambon, and Bernardo Uribe for useful conversations that helped improving this work.

\bibliographystyle{amsplain}

\end{document}